\documentclass{amsart}

\newtheorem{thm}{Theorem}[section]
\newtheorem{prop}[thm]{Proposition}
\newtheorem{lem}[thm]{Lemma}
\newtheorem{cor}[thm]{Corollary}
\newtheorem{dfn}[thm]{\bf Definition}
\newtheorem{rem}[thm]{\bf Remark}
\numberwithin{equation}{section}

\author{Javad Rastegari}
\address{Department of Mathematics\\
University of Western Ontario\\
London, Canada}
\email{jrastega@uwo.ca}

\author{Gord Sinnamon}
\address{Department of Mathematics\\
University of Western Ontario\\
London, Canada}
\email{sinnamon@uwo.ca}
\thanks{Supported by the Natural Sciences and Engineering Research Council of Canada}


\begin{document}

\title{Fourier series in weighted Lorentz spaces}

\begin{abstract} The Fourier coefficient map is considered as an operator from a weighted Lorentz space on the circle to a weighted Lorentz sequence space. For a large range of Lorentz indices, necessary and sufficient conditions on the weights are given for the map to be bounded. In addition, new direct analogues are given for known weighted Lorentz space inequalities for the Fourier transform. Applications are given that involve Fourier coefficients of functions in LlogL and more general Lorentz-Zygmund spaces. \end{abstract}
\maketitle
\section{Introduction}

The study of weighted Fourier inequalities has so far focused on the Fourier transform of functions on $\mathbb{R}$. Very little, for general weights, has been accomplished for the Fourier coefficient map, $f\mapsto\hat f$, where
\[
\hat f(n)=\int_0^1 e^{-2\pi in x}f(x)\,dx,\quad f\in L^1(\mathbb{T}),
\]
for each $n\in \mathbb{Z}$. Although there are many similarities between this map and the Fourier transform on $\mathbb{R}$, the compactness of the domain of $f$ and the discreteness of the domain of $\hat f$ make the theory substantially different.

The Fourier transform and Fourier coefficient map are of fundamental importance in harmonic analysis. Following the successful characterization of weighted Lebesgue-space inequalities for the Hilbert transform and related operators, B. Muckenhoupt proposed the characterization of weighted Lebesgue-space Fourier inequalities as an important goal for the field. From subsequent work by J. Benedetto, H. Heinig, and R. Johnson in \cite{BH1,BHJ} and particularly in \cite{BH2}, weighted Lorentz-space Fourier inequalities emerged as a powerful technique for proving weighted Lebesgue-space Fourier inequalities, as well as being significant in their own right. Here we extend and adapt work from \cite{ECF,FTLS,japan} to give weighted Lorentz-space Fourier inequalities for the Fourier coefficient map. Our results include, for a large range of indices $p$ and $q$, necessary and sufficient conditions on weights $u$ and $w$ for which the inequality,
\begin{equation}\label{GammaGamma}
\|\hat f\|_{\Gamma_q(u)}\le C\|f\|_{\Gamma_p(w)},\quad f\in L^1(\mathbb{T}),
\end{equation}
holds for some constant $C$ independent of $f$. It also includes direct analogues, for the Fourier coefficient map, of the Lorentz-norm Fourier inequalities given in \cite{BH2}.

The Lorentz $\Gamma$-spaces in the inequality above are defined at the end of this introduction, along with the more classical $\Lambda$-spaces and the $\Theta$-spaces that figure prominently in our weight characterization. Section \ref{QC} is devoted to proving a representation theorem for certain generalized quasi-concave functions that will be needed to prove sufficient conditions for (\ref{GammaGamma}). Section \ref{NC} contains a rather technical construction of the test functions that give necessary conditions for (\ref{GammaGamma}). In Section \ref{WC} these are combined to prove our main results: Propositions \ref{sufficient series} and \ref{CC} give a number of conditions that are sufficient to imply (\ref{GammaGamma}) for various index ranges. Theorems \ref{NandS}--\ref{necessary sufficient averaging omega_z series} show that in a certain index range these conditions are also necessary. Applications to Fourier inequalities between Lorentz spaces with specific weights appear in Section \ref{App}.

Our Fourier inequalities are stated and proved for $f\in L^1(\mathbb{T})$ only. Here the circle $\mathbb{T}$ is identified with the real interval $[0,1]$. In most cases they involve a norm or quasi-norm in which the $L^1$ functions are dense. In such a cases it is standard to extend the Fourier coefficient map so that the inequality remains valid. This is left to the reader. 

We will make use of the Fourier coefficient map's well-known group invariance properties: If $f\in L^1(\mathbb{T})$,  $g(x) = e^{2\pi in_0 x} f(x)$ and $h(x) = f(x-x_0)$, then 
\begin{equation}\label{Fourier translation}\hat g(n) = \hat f(n-n_0)\quad\text{and}\quad \hat h(n) = e^{-2\pi inx_0}\hat f(n)
\end{equation}
for any $n_0\in\mathbb{Z}$ and any $x_0\in\mathbb{T}$. 

Throughout the paper, $L^+$ denotes the collection of non-negative Lebesgue measurable functions on $(0,\infty)$ and, for $0<p<\infty$ and $w\in L^+$, the weighted Lebesgue space $L^p(w)$ denotes the normed (or quasi-normed) space of Lebesgue measurable functions $h$ on $(0,\infty)$ for which
\[
\|h\|_{p,w}=\bigg(\int_0^\infty|h(t)|^pw(t)\,dt\bigg)^{1/p}<\infty.
\]
The Lebesgue spaces, for $1\le p\le \infty$, over a general measure space $(X,\mu)$, are defined in the usual way and denoted by $L^p_\mu$.

If $\{(a_i,b_i),i\in I\}$ is a (necessarily finite or countable) collection of disjoint subintervals of $(0,\infty)$ we define the averaging operator $A$ by,
\begin{equation}\label{ave}
Af(x)=\begin{cases}\frac1{b_i-a_i}\int_{a_i}^{b_i}f(t)\,dt,&x\in(a_i,b_i),\\
f(x),&x\notin\cup_{i\in I}(a_i,b_i).\end{cases}
\end{equation}
The class of all such operators $A$ is denoted $\mathcal{A}$. It is an easy exercise to show that each $A\in\mathcal{A}$ maps $L^+$ to $L^+$ and is formally self-adjoint, that is, for all $f,g\in L^+$,
\[
\int_0^\infty Af(t)g(t)\,dt=\int_0^\infty f(t)Ag(t)\,dt.
\]

For $u\in L^+$, $u^o$ denotes the {\it level function} of $u$ with respect to Lebesgue measure on $(0,\infty)$. If the function $\int_0^xu(t)\,dt$ lies under any line on $(0,\infty)$ then it has a well-defined least concave majorant. This majorant is absolutely continuous and so may be represented as an integral. The function $u^o$ is defined by the requirement that $\int_0^x u^o(t)\,dt$ be the least concave majorant of $\int_0^x u(t)\,dt$. We may take $u^o$ to be decreasing, since it necessarily agrees with a decreasing function almost everywhere. Moreover, if $u$ is decreasing then $u^o=u$.

If $\int_0^xu(t)\,dt$ does not lie under any line on $(0,\infty)$ the level function $u^o$ may be defined as the limit of an increasing sequence of level functions. For details, and additional properties of the level function, see \cite{CCDS,LRI,ECF,FTLS,TMWNI,MBFS}.

\subsection{Lorentz Spaces}\label{LS} 
Let $(X, \mu)$ be a $\sigma$-finite measure space. For $f\in L^1_\mu+L^\infty_\mu$, the rearrangement, $f^*$, of $f$ with respect to $\mu$ (see \cite{BS}) is the generalized inverse of the distribution function 
\[
\mu_f(\lambda)=\mu\{x\in X:|f(x)|>\lambda\},
\]
defined by
\[
f^*(t)=\inf\{\lambda>0:\mu_f(\lambda)\le t\}.
\]
The rearrangement is a non-negative, decreasing, Lebesgue measurable function on $(0,\infty)$, and so is
\[
f^{**}(t) =\frac1t\int_0^t f^*.
\]
Define
\[
\|f\|_{\Lambda_p(w)}=\|f^*\|_{p,w},\ 
\|f\|_{\Theta_p(w)}=\sup_{h^{**}\le f^{**}}\|h^*\|_{p,w}, \text{ and }
\|f\|_{\Gamma_p(w)}=\|f^{**}\|_{p,w}.
\]
Here $h$ is a function on $(0,\infty)$ and $h^*$ is its rearrangement with respect to Lebesgue measure. Since $f^*\le f^{**}$ it is easy to verify that, for every $f$,
\begin{equation}\label{LTG}
\|f\|_{\Lambda_p(w)}\le\|f\|_{\Theta_p(w)}\le\|f\|_{\Gamma_p(w)}.
\end{equation}
Define $\Lambda_p(w)$ to be the set of $\mu$-measurable functions $f$ for which $\|f\|_{\Lambda_p(w)}$ is finite, and define $\Theta_p(w)$ and $\Gamma_p(w)$ correspondingly. Clearly, $\Gamma_p(w)\subseteq \Theta_p(w)\subseteq \Lambda_p(w)$.

When $1<p<\infty$, $\|\cdot\|_{\Theta_p(w)}$ and $\|\cdot\|_{\Gamma_p(w)}$ are norms for any non-trivial weight $w\in L^+$ (see \cite{japan}) and $\|\cdot\|_{\Lambda_p(w)}$ is a norm whenever $w$ is decreasing. However, if there exists a constant $c$ such that 
\begin{equation}\label{AMH}
\|f\|_{\Gamma_p(w)}\le c\|f\|_{\Lambda_p(w)}
\end{equation}
for all $\mu$-measurable $f$, then $\|\cdot\|_{\Lambda_p(w)}$ is equivalent to both of the norms, $\|\cdot\|_{\Theta_p(w)}$ and $\|\cdot\|_{\Gamma_p(w)}$. According to \cite{AM} such a $c$ exists whenever $w\in B_p$, that is, whenever there exists a constant $b_p(w)$ such that 
\[
\int_t^\infty \frac{w(s)}{s^p}\,ds\le \frac {b_p(w)}{t^p}\int_0^t w(s)\,ds,\quad t>0.
\]
When $p=1$ the situation is different. If $w\in B_{1,\infty}$, that is, if there exists a constant $b_1(w)$ such that
\[
\frac1y\int_0^yw(t)\,dt\le \frac{b_1(w)}x\int_0^xw(t)\,dt,\quad0<x<y<\infty,
\]
then
$\|\cdot\|_{\Lambda_1(w)}$ is equivalent to the norm $\|\cdot\|_{\Theta_1(w)}=\|\cdot\|_{\Lambda_1(w^o)}$. This follows from Lemma 2.2 and Lemma 2.5 of \cite{FTLS}.

When the underlying measure $\mu$ is Lebesgue measure, or any other infinite non-atomic measure, $\|\cdot\|_{\Lambda_p(w)}$ is a norm if and only if $w$ is decreasing. Also, the $B_p$ condition (when $p>1$) and the $B_{1,\infty}$ condition (when $p=1$) are necessary and sufficient for $\|\cdot\|_{\Lambda_p(w)}$ to be equivalent to a norm. See \cite{Sa} and \cite{CGS}.

If the underlying measure $\mu$ is finite, and $f\in L^1_\mu+L^\infty_\mu=L^1_\mu$, then $f^*$ is supported on $(0,\mu(X))$. Thus, $\|f\|_{\Lambda_p(w)}$ depends only on the restriction of $w$ to $(0,\mu(X))$. However, $f^{**}$ is not supported on $(0,\mu(X))$ so $\|f\|_{\Gamma_p(w)}$ does depend on values of $w$ outside $(0,\mu(X))$, but only through the value of $\int_{\mu(X)}^\infty w(t)\frac{dt}{t^p}$. 

If the underlying measure $\mu$ is counting measure on $\mathbb Z$, the space $L^1_\mu+L^\infty_\mu$ may be identified with a space of sequences. It is possible to define the rearrangement of a sequence directly to obtain another sequence but we will stick with the above definition, viewing $L^1_\mu+L^\infty_\mu$ as a space of $\mu$-measurable functions, with decreasing functions on $(0,\infty)$ as their rearrangements. This is only a notational difference; the decreasing functions we obtain are constant on the intervals, $[n,n+1)$ for $n=0,1,\dots$ so each may be identified with the corresponding rearranged sequence if desired.

\section{Quasi-Concave Functions}\label{QC}

Functions with two monotonicity conditions arise naturally in our study of Fourier series in Lorentz spaces. Let $\alpha+\beta>0$. By $\Omega_{\alpha,\beta}$ we mean the collection of all functions  $f\in L^+$ such that $x^{\alpha}f(t)$ is increasing and $x^{-\beta}f(t)$ is decreasing. Notice that $\Omega_{\alpha,\beta}$ is a cone, being closed under addition and under multiplication by positive scalars. Functions in $\Omega_{0,1}$ are called {\it quasi-concave} because they are equivalent to concave functions, and functions in $\Omega_{\alpha,\beta}$ are called {\it generalized quasi-concave} functions. 

Our chief interest will be in the cone $\Omega_{2,0}$, but we begin by looking at all the cones $\Omega_{\alpha,\beta}$ together because they are related by simple transformations.  For instance, if $\lambda>0$ and $g(t)=t^\gamma f(t^{1/\lambda})$, then $g \in \Omega_{\alpha,\beta}$ if and only if $f \in \Omega_{\lambda(\alpha+\gamma),\lambda(\beta-\gamma)}$. 

Besides being in $\Omega_{2,0}$, the functions we encounter are constant on the interval $(0,1)$. To deal with this additional restriction in general terms we introduce the cones,
\[
P_\xi^r=\{f\in L^+:t^{-r}f(t)\text{ is constant on }(0,\xi)\}
\]
and set $P=P_1^0$.  

\begin{dfn}
Maps  $A : L_{\nu}^+ \to L_{\mu}^+$  and $B : L_{\mu}^+ \to L_{\nu}^+$ are called formal adjoints provided
\[
\int_Y Af(y) g(y) \, d\mu(y) = \int_X f(x)Bg(x) \, d\nu(x)
\]
for all $f\in L_{\nu}^+$ and $g\in L_{\mu}^+$.
\end{dfn}
The following lemma is a modification of Lemma 4 in \cite{japan}. The lemma was applied outside its scope in Theorem 6 of \cite{japan}. In the version below we widen the scope to include all operators $A$ having formal adjoints. This includes the averaging operators introduced in (\ref{ave}) and fills the gap in the proof of Theorem 6 of \cite{japan}.
\begin{lem}\label{J.4} Let $0 < p \le 1 \le q < \infty$. Suppose $(Y, \mu)$, $(X, \nu)$, $(T, \lambda)$ are $\sigma$-finite measure spaces, $k(x, t)\ge0$ is a  $\nu\times\lambda$-measurable function, and $A : L_\nu^+ \to  L_\mu^+$ has a formal adjoint. Define $K$ by $Kh(x) = \int_T k(x, t) h(t) \, d\lambda(t)$ and let $k_t (x) = k(x,t)$. Then, for any $u\in L^+_{\mu}$ and $v\in L^+_\nu$,
\begin{equation}\label{J.4.h}
\sup_{h\ge 0} \dfrac{\|AKh\|_{q,u\mu}}{\|Kh\|_{p,v\nu}} \le  \operatornamewithlimits{ess\,sup}_{t\in T} \frac{\|Ak_t\|_{q,u\mu}}{\|k_t\|_{p,v\nu}}.
\end{equation}
\end{lem}
\begin{proof} Let $C$ be the right-hand side of (\ref{J.4.h}) and fix $h\in L^+_\lambda$. Since $0<p\le1$, Minkowski's integral inequality shows that
\begin{align*}
\int_T\|k_t\|_{p,v\nu}h(t)\,d\lambda(t)
&=\int_T\bigg(\int_Xk(x,t)^pv(x)\,d\nu(x)\bigg)^{1/p}h(t)\,d\lambda(t)\\
&\le\bigg(\int_X\bigg(\int_Tk(x,t)h(t)
\,d\lambda(t)\bigg)^pv(x)\,d\nu(x)\bigg)^{1/p}\\
&=\|Kh\|_{p,v\nu}.
\end{align*}

Let $B:L_\mu^+ \to  L_\nu^+$ be a formal adjoint of $A$. For any $g\in L^+_\mu$ with $\|g\|_{q',u\mu}\le1$, Tonelli's theorem implies,
\begin{align*}
\int_YAKh(y)g(y)u(y)\,d\mu(y)&=\int_XKh(x)B(gu)(x)\,d\nu(x)\\
&=\int_T\bigg(\int_Xk_t(x)B(gu)(x)\,d\nu(x)\bigg) h(t)\,d\lambda(t).
\end{align*}
But, by H\"older's inequality,
\[
\int_Xk_t(x)B(gu)(x)\,d\nu=\int_YAk_t(y)g(y)u(y)\,d\mu(y)\le\|Ak_t\|_{q,u\mu}\le C\|k_t\|_{p,v\nu},
\]
for $\lambda$-almost every $t$. Therefore,
\[
\int_Y (AKh)(y)g(y)u(y)\,d\mu(y)
\le C\int_T\|k_t\|_{p,v\nu}h(t)\,d\lambda(t)\le C\|Kh\|_{p,v\nu}.
\]
Taking the supremum over all such $g$ yields
\[
\|AKh\|_{q,u\mu}\le C\|Kh\|_{p,v\nu}.
\]
Since $h\in L^+_\lambda$ was arbitrary, the conclusion follows.
\end{proof}

In order to apply this result to $P_\xi^\beta\cap\Omega_{\alpha,\beta}$, we show that the range of a certain positive operator is a large subset of this cone. The positive operator is $K_{\xi}^{\alpha,\beta}$, defined by
\[
K_{\xi}^{\alpha,\beta}h(x)=\int_\xi^\infty k^{\alpha,\beta}(x,t)h(t)\,dt,
\]
where $k^{\alpha,\beta}(x,t)=\min(x^\beta t^{-\alpha},x^{-\alpha} t^\beta)$. It is easy to check that for fixed $t$, $k^{\alpha,\beta}_t(x)$ is in $\Omega_{\alpha,\beta}$ and that $K_{\xi}^{\alpha,\beta}h(x)\in P_\xi^\beta\cap\Omega_{\alpha,\beta}$ whenever $h\in L^+$. Thus, the image of $L^+$ under $K_{\xi}^{\alpha,\beta}$ is a subset of $ P_\xi^\beta\cap\Omega_{\alpha,\beta}$. What we mean by ``large subset'' is in Lemma \ref{K approximate series}. 

We start with a lemma stating the geometrically obvious fact that if a function is linear on some interval, so is its least concave majorant.
\begin{lem}\label{lcm line segment} Suppose $\tilde g$ is the least concave majorant of $g\in L^+$. If $\xi>0$, $c\ge0$ and $g(x) = cx$ on $(0,\xi)$ then $\tilde g(x) = x\tilde g(\xi)/\xi$ on $(0,\xi)$.
\end{lem}
\begin{proof} Let $\lambda=\tilde g(\xi)/\xi$. Since $\tilde g\ge0$ is concave, $\lambda x\le \tilde g(x)$ on $(0,\xi]$ and $\lambda x\ge \tilde g(x)$ on $[\xi,\infty)$. Since $\tilde g$ is continuous, 
$$
\lambda = \tilde g(\xi)/\xi
=\lim_{x\to \xi-}\tilde g(x)/x
\ge\lim_{x\to \xi-}g(x)/x=c.
$$
Thus $\lambda x\ge cx=g(x)$ on $(0,\xi)$ and $\lambda x\ge\tilde g(x)\ge g(x)$ on $[\xi,\infty)$. So $\lambda x$ is a concave majorant of $g$ and therefore $\lambda x\ge\tilde g(x)$ on $(0,\infty)$. In particular, $\tilde g(x) = \lambda x$ on $(0,\xi)$.
\end{proof}

The next lemma shows that every function in $P_\xi^1\cap\Omega_{0,1}$ is equal, up to equivalence, to the limit of an increasing sequence of functions in the range of $K^{0,1}_{\xi}$.
\begin{lem}\label{approximate concave series}  Let $\xi\ge 0$ and let $g$ be a quasi-concave function such that $g(x)/x$ is constant on $(0,\xi)$. If $\tilde g$ is the least concave majorant of $g$, then $g \le \tilde g\le 2g$ and there exists a sequence of functions $\ell_n\in L^+$ such that $K^{0,1}_{\xi}\ell_n$ increases to $\tilde g$ pointwise.
\end{lem}
\begin{proof} Proposition 2.5.10 of \cite{BS} shows that $g\le\tilde g\le 2g$. 

Recall that a concave function on $(0,\infty)$ is absolutely continuous on closed subintervals of $(0,\infty)$. It has left and right derivatives everywhere, the right derivative is right continuous, both are decreasing, and the right derivative is less than or equal to the left derivative at each point.

Let $\varphi$ denote the right derivative of $\tilde g$ and let $a=\tilde g(\xi) -\xi \varphi(\xi)$ if $\xi>0$ and $a=\tilde g(0+)$ if $\xi=0$.   If $\xi=0$ it is clear that $a\ge0$ and if $\xi>0$, Lemma \ref{lcm line segment} shows that $\tilde g(\xi)/\xi$ is the left derivative of $\tilde g$ at $\xi$ so in this case, too, $a\ge0$. Let $\varphi(\infty)=\lim_{t\to\infty}\varphi(t)$. For $n>\xi$ and $t>0$, set
\begin{equation}\label{hdefn}
\ell_n(t)=\varphi(\infty)\chi_{(n,n+1)}(t)+(a/t)n\chi_{(\xi,\xi+\frac1n)}(t)
+\frac{\varphi(t)-\varphi(t\frac{n+1}n)}{t\log(\frac{n+1}n)}.
\end{equation}
Since $\varphi$ is decreasing, $\ell_n\in L^+$ for all positive integers $n>\xi$.

We apply $K^{0,1}_{\xi}$ to each of the three terms separately. The first term becomes $\int_n^{n+1}\varphi(\infty)\min(x,t)\,dt$. For each $x$, this is a moving average of the increasing function $\varphi(\infty)\min(x,t)$ and is therefore increasing with $n$. It converges to $x\varphi(\infty)$.

The second term becomes $n\int_\xi^{\xi+1/n}(a/t)\min(x,t)\,dt$. This is a shrinking average of the decreasing function $a\min(x/t,1)$ and is therefore increasing with $n$. It converges to $a\min(x/\xi,1)$.

Let the third term of (\ref{hdefn}) be $\bar \ell_n(t)$. For $y>0$,
\begin{align*}
\int_y^\infty \bar \ell_n(t)\,dt
&=\frac1{\log(\frac{n+1}n)}
\lim_{M\to\infty}\bigg(\int_y^M\varphi(t)\,\frac{dt}t-\int_y^M\varphi(t\tfrac{n+1}n)\,\frac{dt}t\bigg)\\
&=\frac1{\log(\frac{n+1}n)}
\lim_{M\to\infty}\bigg(\int_y^M\varphi(t)\,\frac{dt}t-\int_{y\tfrac{n+1}n}^{M\tfrac{n+1}n}\varphi(t)\,\frac{dt}t\bigg)\\
&=\frac{\int_y^{y\tfrac{n+1}n}\varphi(t)\,\frac{dt}t}{\int_y^{y\tfrac{n+1}n}\,\frac{dt}t}
-\lim_{M\to\infty}\frac{\int_M^{M\tfrac{n+1}n}\varphi(t)\,\frac{dt}t}{\int_M^{M\tfrac{n+1}n}\,\frac{dt}t}\\
&=\frac{\int_y^{y\tfrac{n+1}n}\varphi(t)\,\frac{dt}t}{\int_y^{y\tfrac{n+1}n}\,\frac{dt}t}-\varphi(\infty).
\end{align*}
Since $\varphi$ is decreasing and right continuous, its shrinking average on $(y,y\tfrac{n+1}n)$ increases with $n$ and converges to $\varphi(y)$. Thus the last expression increases to $\varphi(y)-\varphi(\infty)$ as $n\to\infty$. The identity,
\[
K^{0,1}_{\xi}h(x)=\int_\xi^\infty\min(x,t)h(t)\,dt=\int_0^x\int_{\max(y,\xi)}^\infty h(t)\,dt\,dy,
\]
shows that $K^{0,1}_{\xi}\bar \ell_n$ also increases with $n$ and, by the Monotone Convergence Theorem,
\begin{align*}
\lim_{n\to\infty}K^{0,1}_{\xi}\bar \ell_n(x)
&=\lim_{n\to\infty}\int_0^x\int_{\max(y,\xi)}^\infty\bar \ell_n(t)\,dt\,dy\\
&=\int_0^x \varphi(\max(y,\xi))-\varphi(\infty)\,dy\\
&=\begin{cases}x\varphi(\xi)-x\varphi(\infty),&0<x<\xi;\\
\xi \varphi(\xi)+\tilde g(x)-\tilde g(\xi)-x\varphi(\infty),&0<\xi\le x;\\
\tilde g(x)-\tilde g(0+)-x\varphi(\infty),&0=\xi< x.
\end{cases}
\end{align*}
Combining the three terms of (\ref{hdefn}), we conclude that $K^{0,1}_{\xi}\ell_n$ increases with $n$. When $0<x<\xi$ the limit is, 
\[
x\varphi(\infty)+a(x/\xi)+x\varphi(\xi)-x\varphi(\infty)=x\tilde g(\xi)/\xi=\tilde g(x)
\]
by Lemma \ref{lcm line segment}. When $0<\xi\le x$ the limit is 
\[
x\varphi(\infty)+a+\xi \varphi(\xi)+\tilde g(x)-\tilde g(\xi)-x\varphi(\infty)=\tilde g(x),
\]
and when $0=\xi< x$ it is
\[
x\varphi(\infty)+a+\tilde g(x)-\tilde g(0+)-x\varphi(\infty)=\tilde g(x).
\]
This completes the proof.
\end{proof}

This approximation of quasi-concave functions can be used to give a similar result for functions in $P_{\xi}^{\beta} \cap \Omega_{\alpha,\beta}$. They can be realized as increasing limits of functions of type $K_{\xi}^{\alpha,\beta}h$, up to equivalence. This result extends Lemma 5 of \cite{japan}.

\begin{lem}\label{K approximate series} Suppose $\xi\ge0$, and $\alpha,\beta\in \mathbb{R}$ satisfy $\alpha+\beta>0$. If $f\in P_{\xi}^{\beta}\cap \Omega_{\alpha,\beta}$, then there exists $\tilde{f} \in L^+$ and a sequence of functions  $\{h_n\}$ in $L^+$ such that $f \le \tilde{f}\le 2f$ and $K_{\xi}^{\alpha,\beta} h_n$ increases to $\tilde{f}$ pointwise.
\end{lem}
\begin{proof} Within this proof let $y=x^{\alpha+\beta}$ and $s=t^{\alpha+\beta}$. Fix $f\in P_{\xi}^{\beta}\cap \Omega_{\alpha,\beta}$ and define $g$ by setting $g(y)=x^\alpha f(x)$. This ensures that $g\in P_{\xi^{\alpha+\beta}}^1\cap \Omega_{0,1}$. Let $\tilde g$ be the least concave majorant of the quasi-concave function $g$ and apply Lemma \ref{approximate concave series}, with $\xi$ replaced by $\xi^{\alpha+\beta}$, to obtain functions $\ell_n\in L^+$ such that $K_{\xi^{\alpha+\beta}}^{0,1} \ell_n$ increases to $\tilde g$ as $n\to\infty$. 

Set $\tilde f(x)=x^{-\alpha}\tilde g(y)$. Since $g \le \tilde g\le 2g$ we also have
$f \le \tilde{f}\le 2f$. (Note that $\tilde f$ is not the least concave majorant of $f$ in general.) Then define $h_n$ by requiring that $t^{-\alpha}h_n(t)\,dt=\ell_n(s)\,ds$. Evidently, $h_n\in L^+$. Also, $K_{\xi}^{\alpha,\beta} h_n(x)$ is equal to
\[
\int_\xi^\infty\min(x^\beta t^{-\alpha},x^{-\alpha}t^\beta)h_n(t)\,dt
=x^{-\alpha}\int_{\xi^{\alpha+\beta}}^\infty\min(y,s)\ell_n(s)\,ds
=x^{-\alpha}K_{\xi^{\alpha+\beta}}^{0,1}\ell_n(y).
\]
Therefore, $K_{\xi}^{\alpha,\beta} h_n(x)$ is increasing with $n$ and converges to $x^{-\alpha}\tilde g(y)=\tilde f(x)$. This completes the proof.
\end{proof}

Now we have all the machinery to prove the main result of this section. 

\begin{prop}\label{averaging operator norm series} Suppose $\xi\ge0$, and $\alpha,\beta\in \mathbb{R}$ satisfy $\alpha+\beta>0$. Let $0<p\le 1\le q < \infty$ and $u,v\in L^+$. If $A\in\mathcal{A}$, then 
\begin{equation}\label{aons}
 \sup_{t > \xi} \dfrac{\|A k^{\alpha,\beta}_t\|_{q,u}}{\|k^{\alpha,\beta}_t\|_{p,v}} \le\sup_{f \in  P_{\xi}^{\beta}\cap\Omega_{\alpha,\beta}} \dfrac{\|Af\|_{q,u}}{\|f\|_{p,v}}\le 2  \sup_{t > \xi} \dfrac{\|A k^{\alpha,\beta}_t\|_{q,u}}{\|k^{\alpha,\beta}_t\|_{p,v}}.
\end{equation}
\end{prop}
\begin{proof} The definition, $k^{\alpha,\beta}_t(x)=\min(x^\beta t^{-\alpha},x^{-\alpha}t^\beta)$, ensures that if $t>\xi$ then $k_t^{\alpha,\beta}\in P_{\xi}^{\beta}\cap\Omega_{\alpha,\beta}$. This proves the first inequality of (\ref{aons}). For the other, let
\[
D= \sup_{t > \xi} \dfrac{\|A k^{\alpha,\beta}_t\|_{q,u}}{\|k^{\alpha,\beta}_t\|_{p,v}}.
\]
We apply Lemma \ref{J.4} taking $\mu$ and $\nu$ to be Lebesgue measure on $(0,\infty)$. Note that each $A\in \mathcal{A}$ is its own formal adjoint. Let $\lambda$ be Lebesgue measure on $(\xi,\infty)$, $K=K^{\alpha,\beta}_\xi$ and $k(x,t)=k^{\alpha,\beta}(x,t)$. The conclusion is that for all $h\in L^+$,
\[
\|AK_\xi^{\alpha,\beta}h\|_{q,u}\le D\|K_\xi^{\alpha,\beta}h\|_{p,v}.
\]
Now fix $f\in P_{\xi}^{\beta}\cap\Omega_{\alpha,\beta}$ and apply Lemma \ref{K approximate series} to get an increasing sequence $h_n\in L^+$ such that $f\le\lim_{n\to\infty}K_\xi^{\alpha,\beta}h_n\le 2f$. The averaging operator $A$  preserves order, so, by the Monotone Convergence Theorem,
\[
\|Af\|_{q,u}\le\lim_{n\to\infty}\|A K_\xi^{\alpha,\beta} h_n\|_{q,u}
\le D\lim_{n\to\infty}\|K_\xi^{\alpha,\beta} h_n\|_{p,v}\le2D\|f\|_{p,v}.
\]
This proves the second inequality of (\ref{aons}).
\end{proof}

One consequence of Proposition \ref{averaging operator norm series} is the following extension of Theorem 1 in \cite{M2} (see also Theorem 3 in \cite{M1}) from the range $1\le p \le q <\infty$ to $0<p \le q <\infty$. It is possible to get such an extension directly from Maligranda's Theorem 1 beginning with the case $p=1$ and making the substitution $f\mapsto f^p$. By this method one obtains the constant $2^{1/p}$ in place of the smaller $2^{1/q}$ that appears below.
 
\begin{prop}\label{maligranda series} Suppose $\xi\ge0$, and $\alpha,\beta\in \mathbb{R}$ satisfy $\alpha+\beta>0$. If $0<p \le q <\infty$ and $u,v \in L^+$, then
\begin{equation}\label{eq maligranda series}
\sup_{z> \xi}  \dfrac{\| k^{\alpha ,\beta}_z \|_{q,u}}{\|k^{\alpha,\beta}_z \|_{p,v}}\le\sup_{f\in P_{\xi}^{\beta}\cap\Omega_{\alpha,\beta}} \dfrac{\| f \|_{q,u}}{\| f \|_{p,v}}  \le 2^{1/q} \sup_{z> \xi}  \dfrac{\| k^{\alpha ,\beta}_z \|_{q,u}}{\|k^{\alpha,\beta}_z \|_{p,v}}
\end{equation}
\end{prop}
\begin{proof} Taking $g=f^q$, it is routine to verify that,
\[
\sup_{f\in P_{\xi}^{\beta}\cap\Omega_{\alpha,\beta}} \dfrac{\| f \|_{q,u}}{\| f \|_{p,v}} =\bigg(\sup_{g\in P_{\xi}^{q\beta}\cap\Omega_{q\alpha,q\beta}} \dfrac{\| g \|_{1,u}}{\| g \|_{p/q,v}}\bigg)^{1/q}.
\]
But, $(k_t^{\alpha,\beta}(x))^q=k_t^{q\alpha,q\beta}(x)$, so we also have,
\[
\sup_{t> \xi}  \dfrac{\| k^{\alpha ,\beta}_t \|_{q,u}}{\|k^{\alpha,\beta}_t \|_{p,v}}=\bigg( \sup_{t> \xi}  \dfrac{\| k^{q\alpha ,q\beta}_t \|_{1,u}}{\|k^{q\alpha,q\beta}_t \|_{p/q,v}}\bigg)^{1/q}.
\]
The result now follows from Proposition \ref{averaging operator norm series}, with indices $p/q$ and $1$, by taking $A$ to be the identity.
\end{proof} 

We end this section by stating the special cases of Propositions \ref{averaging operator norm series} and \ref{maligranda series} that will be used for our results in inequalities for Fourier series. Recall that $P=P_1^0$ and $\omega_z(x)=\min(z^{-2},x^{-2})=k^{2,0}_z(x)$.

\begin{cor}\label{averaging operator norm series fourier}  Let $0<p\le 2\le q < \infty$ and $u,v\in L^+$. If $A\in\mathcal{A}$, then 
\[
\sup_{z >1} \dfrac{\|A\omega_z\|_{q/2,u}}{\|\omega_z\|_{p/2,v}} \le\sup_{f \in  P\cap\Omega_{2,0}} \dfrac{\|Af\|_{q/2,u}}{\|f\|_{p/2,v}} \le 2\sup_{z >1} \dfrac{\|A\omega_z\|_{q/2,u}}{\|\omega_z\|_{p/2,v}}.
\]
\end{cor}

\begin{cor}\label{maligranda series fourier}  Let $0<p \le q <\infty$ and $u,v\in L^+$. Then 
\[
\sup_{z >1} \dfrac{\|\omega_z\|_{q/2,u}}{\|\omega_z\|_{p/2,v}} \le\sup_{f \in  P\cap\Omega_{2,0}} \dfrac{\|f\|_{q/2,u}}{\|f\|_{p/2,v}} \le 2^{2/q}\sup_{z >1} \dfrac{\|\omega_z\|_{q/2,u}}{\|\omega_z\|_{p/2,v}}. 
\]
\end{cor}

\section{Necessary Conditions}\label{NC}

Here we construct the test functions that produce our necessary condition for the Fourier inequality 
\begin{equation}\label{GammaLambda}
\|\hat f\|_{\Lambda_q(u)}\le C\|f\|_{\Gamma_p(w)},\quad f\in L^1(\mathbb{T}).
\end{equation}
The condition we obtain is automatically necessary for the stronger inequality (\ref{GammaGamma}) as well. The method is similar to the construction given for the Fourier transform in \cite{FTLS}, in that one test function is constructed for each averaging operator in the class $\mathcal A$, see (\ref{ave}), and each value of a positive real parameter $z$. The details of construction in the Fourier series case are quite different, however, because of the finite measure on $\mathbb{T}$ and the atomic measure on $\mathbb{Z}$.

The idea is to take advantage of the large class of functions $g$ whose rearrangements coincide with $f^*$, for a given $f$. It turns out that there is enough freedom within this class to ensure that the rearrangement $\hat g^*$, of the Fourier series of $g$, possesses the properties we require. The first four lemmas are needed to give the main construction in Lemma \ref{test fun z>1}. The general necessary condition is proved in Theorem \ref{necessary series}.

Throughout this section we use $\mu$ to denote counting measure on $\mathbb{Z}$. 

The Fourier series of the characteristic function of an interval is easy to calculate. The first lemma gives an estimate of its rearrangement. 
\begin{lem}\label{1/z fourier star est}
Suppose $z \ge 3$ and let $f(x)=\chi_{(0,1/z)}(x)$, viewed as a  function on $\mathbb{T}$. Then $\hat{f}^*(y) \ge 1/(3\pi y+9\pi z)$.
\end{lem}
\begin{proof}
The Fourier coefficients of $f$ may be computed directly. If $k\ne0$, then
\[
\hat f(k)=\int_0^1 e^{-2\pi ikx}f(x)\,dx=e^{-ik\pi/z}\frac{\sin(k\pi/z)}{k\pi}.
\] 
For $\alpha>0$, let 
\[
E_\alpha=\{k\in \mathbb{Z}: |\hat{f}(k)| > \alpha \}
\supseteq\{k\in \mathbb{Z}\setminus\{0\} : |\sin(k\pi/z)| > \alpha|k|\pi\}.
\] To estimate $\mu(E_\alpha)$, the number of elements in $E_\alpha$, note that any real interval of length $L$ contains at least $L-1$ integers. We will also make use of the following simple estimate based on the convexity of the sine function: If $n$ is an integer and $|x-(2n-1)|\le1$, then $|\sin(\pi x/2)|\ge1-|x-(2n-1)|$.

Let $N$ be the greatest integer less than $1/(\alpha \pi z)$, fix a positive integer $n\le N$ and suppose $k$ is an integer in the open interval of length $z(1-\alpha\pi zn)$ centred at $(z/2)(2n-1)$. Then,
\[
|k-(z/2)(2n-1)|<(z/2)(1-\alpha\pi zn)<z/2
\]
and hence $(n-1)z<k<nz$. Also, 
\[
|2k/z-(2n-1)|<1-\alpha\pi zn\le1,
\]
so the sine function estimate gives,
\[
|\sin(k\pi/z)|\ge1-|2k/z-(2n-1)|>\alpha\pi zn>\alpha\pi k.
\]
If follows that $k\in ((n-1)z,nz)\cap E_\alpha$. Thus, there are at least $z(1-\alpha\pi zn)-1$ positive integers in $((n-1)z,nz)\cap E_\alpha$. Summing these from $n=1\dots N$ shows that $\mu((0,\infty)\cap E_\alpha)$ is not less than,
\[
(z-1)N-z^2\alpha \pi N(N+1)/2
\ge (z-1)\left(\frac1{\alpha\pi z}-1\right)-\frac z2\left(\frac1{\alpha\pi z}+1\right).
\]
Evidently, $k\in E_\alpha$ if and only if $-k\in E_\alpha$. So, using $z\ge3$,
\[
\mu(E_\alpha)\ge2\mu((0,\infty)\cap E_\alpha)\ge\frac{1-2/z}{\alpha\pi}+2-3z
\ge\frac1{3\alpha\pi}-3z.
\]

The definition of the rearrangement ensures that when $\alpha=\hat f^*(y)$, $\mu(E_\alpha)\le y$, so
\[
\frac1{3\hat f^*(y)\pi}-3z\le y
\]
and we have
\[
\hat f^*(y)\ge\frac1{3\pi y+9\pi z}.
\]
\end{proof}

The rearrangement of a characteristic function depends on the measure of the underlying set, but not on its geometry. On the other hand, the Fourier series of a characteristic function is profoundly affected by the geometry of the underlying set. Here we take advantage of this fact to get a dilation-like result that behaves oppositely to what we expect from a Fourier dilation. 
\begin{lem}\label{dilation star est} Let $k$ be a positive integer and $z>1$. Let $f(x)=\chi_{[0,1/(kz))}(x)$. Then for any $\varepsilon >0$ there exists a function $g\in L^1(\mathbb{T})$ such that
\[
g^*(s) = f^*(s/k) \quad and \quad  \hat{g}^*(y) \ge \hat{f}^*(y/k) - \varepsilon
\]
for $0\le s <1$ and $y>0$.
\end{lem}
\begin{proof}
We show that for a sufficiently large integer $M$,
\[
g(x) = \sum_{j=0}^{k-1} e^{2\pi ijMx} f(x-j/(kz))
\]
will be the desired function. First notice that the translates of $f$ in the sum above are supported on disjoint subsets of $[0,1)$. Thus,
\[
|g(x)| =  \sum_{j=0}^{k-1} |e^{2\pi ijMx} f(x-j/(kz))| = \sum_{j=0}^{k-1} f(x-j/(kz)) = \chi_{[0,1/z)}(x).
\]
Furthermore $|g|$ and $f$ are both decreasing so, for $0\le s<1$, 
\[
g^*(s) = |g(s)|=\chi_{[0,1/z)}(s)=f(s/k) = f^*(s/k).
\]
This shows that $g$ satisfies the first conclusion of the lemma no matter what $M$ is chosen.

To establish the second conclusion we make use of the properties (\ref{Fourier translation}) to get,
\[
\hat{g}(n) = \sum_{j=0}^{k-1} e^{-2\pi i(n-jM)j/(kz)} \hat{f}(n-jM).
\] 
Fix $\varepsilon>0$ and choose $M=2k/(\pi\varepsilon)$. For all $n$ satisfying $|n| > M/2$,
\[
|\hat{f}(n)| = \Big|\frac{-e^{in\pi/(kz)}}{n\pi}\sin(n\pi/(kz))\Big|\le \frac{1}{n\pi}<\frac{\varepsilon}{k}.
\]
So if $|n-jM|<M/2$ for some $j$ then, for every $l\neq j$, $|n-lM| > M/2$ and hence $|\hat f(n-lM)|<\varepsilon/k$. It follows that if $|n-jM|<M/2$, then
\[
|\hat{g}(n)|\ge |\hat{f}(n-jM)| - (k-1)\varepsilon/k
\ge |\hat{f}(n-jM)| - \varepsilon.
\]

Now we can estimate the distribution function of $\hat{g}$. For $\alpha >0$,
\begin{align*}
\mu_{\hat{g}}(\alpha) 
& = \mu \{n\in \mathbb{Z}: |\hat{g}(n)| >\alpha \}
\\& \ge \sum_{j=0}^{k-1} \mu \{n\in (jM-M/2, jM+M/2) : |\hat{g}(n)| >\alpha \}
\\&  \ge \sum_{j=0}^{k-1} \mu \{n\in (jM-M/2, jM+M/2) : |\hat{f}(n-jM)| - \varepsilon >\alpha \}
\\& = k \mu \{n\in (-M/2, M/2) : |\hat{f}(n)| >\alpha + \varepsilon \}
\end{align*}

Since $|\hat{f}(n)| < \varepsilon$ when $n\notin (-M/2,M/2)$, 
\[
\mu_{\hat{g}}(\alpha) 
\ge k \mu \{n\in \mathbb{Z}: |\hat{f}(n)| >\alpha + \varepsilon \}
= k\mu_{\hat{f}}(\alpha+\varepsilon).
\]
Now for $y>0$,
\[
y \ge \mu_{\hat{g}}(\hat{g}^*(y)) \ge k\mu_{\hat{f}} (\hat{g}^*(y)+\varepsilon)  
\]
and hence
\[
\hat{f}^*(y/k) \le \hat{f}^*\big(\mu_{\hat f} \big(\hat{g}(y)+\varepsilon \big)\big) \le \hat{g}^*(y)+ \varepsilon,
\]
as required.
\end{proof} 

The last two results combine to give a useful one-parameter family of estimates for the rearrangement of the Fourier series of a characteristic function with an underlying set of fixed measure.
\begin{lem}\label{1/z fourier star epsilon est} For $z\ge 3$, $r>0$ and $\varepsilon >0$ there exists a function $g\in L^1(\mathbb{T})$ such that 
\[
g^* = \chi_{[0,1/z)} \quad and \quad  \hat{g}^*(y) \ge \dfrac{1}{3\pi y/r + 9\pi (r+1)z }-\varepsilon
\]
\end{lem}
\begin{proof}
Let $k$ be the integer satisfying $r\le k < r+1$ and let $f = \chi_{[0,1/kz)}$. Then by Lemma \ref{dilation star est} there exists a $g$ such that,
\[
g^*(s) = f^*(s/k) = \chi_{[0,1/z)} \quad and \quad \hat{g}^*(y) \ge \hat{f}^*(y/k) - \varepsilon
\]
Lemma \ref{1/z fourier star est} yields 
\[
\hat{g}^*(y) \ge  \dfrac{1}{3\pi y/k + 9\pi kz}-\varepsilon \ge \dfrac{1}{3\pi y/r + 9\pi (r+1)z } -\varepsilon
\]
\end{proof} 

The next lemma is a variation of the construction in Lemma \ref{dilation star est}. This time the characteristic function is subdivided into infinitely many parts of differing sizes. The increase in generality is balanced by the coarser estimate obtained for the rearrangement of the Fourier series. 
\begin{lem}\label{seq fourier star est} Let $\{p_j\}$ be a sequence of non-negative real numbers satisfying $\sum_{j=1}^{\infty} p_j = p_0 \le1$. For each $p_j$ let $f_j=\chi_{[0,p_j)}$ be a function on the unit circle. Then for any $\varepsilon>0$ there exists a function $g\in L^1(\mathbb{T})$ such that
\[
g^*=\chi_{[0,p_0)} \quad and \quad \hat{g}^*(y) \ge  \hat{f}_j^*(y) -\varepsilon, \quad j=1,2,\dots.
\]
\end{lem}
\begin{proof}
Let $X_1=0$ and $X_j=\sum_{l=1}^{j-1} p_l$ for $j \ge 2$. Define $g$ by,
\[
g(x) = \sum_{j=1}^{\infty} e^{2\pi iM_j x} f_j(x-X_j),
\]
where the $M_j$, $j=1,2,\dots$ are to be chosen later. The definitions of $X_1,X_2,\dots$ ensure that the translates $f_j(x-X_j)$ have disjoint supports, and that,
\[
|g(x)| = \sum_{j=1}^{\infty} |e^{2\pi iM_j x} f_j(x-X_j) |= \chi_{[0,p_0)}.
\]
Since $|g|$ is decreasing, $g^* = |g|=\chi_{[0,p_0)}$, the first conclusion of the lemma.

The Fourier coefficients of $g$ are given by,
\[
\hat{g}(n) = \sum_{j=1}^{\infty} e^{-2\pi i(n-M_j)X_j} \hat{f_j}(n-M_j).
\]
(Note that since the series defining $g$ converges in $L^1(\mathbb{T})$, the series defining $\hat g$ converges in $L^\infty(\mathbb{Z})$ and hence pointwise.) We choose $M_1,M_2,\dots$ so that the intervals, $I_j=(M_j-2^j/(\pi\varepsilon),M_j+2^j/(\pi\varepsilon))$ are disjoint for $j=1,2,\dots$. This implies that if $n\in I_j$, then $n\not\in I_l$ for $l\ne j$ so,
\[
|\hat{f_l}(n-M_l)|=\left|\frac{\sin((n-M_l)\pi p_l)}{(n-M_l)\pi} \right|\le \frac{1}{|n-M_l|\pi}<\dfrac{\varepsilon}{2^l}.
\]
Thus, for $n\in I_j$,
\[
|\hat{g}(n)| \ge |\hat{f_j} (n-M_j)| - \sum_{l\neq j} \frac{\varepsilon}{2^l} \ge \hat{f_j}(n-M_j) - \varepsilon.
\]
For any $j$ and any $\alpha >0$,
\[
\mu_{\hat{g}}(\alpha)
\ge\mu \{n\in I_j: |\hat{g}(n)| >\alpha \}
\ge\mu \{n\in I_j: |\hat{f_j}(n-M_j)| >\alpha+\varepsilon \}.
\]
But $|\hat{f}_j(n-M_j)| < \varepsilon$ for $n\notin I_j$, so
\[
\mu_{\hat{g}}(\alpha)
\ge\mu \{n\in \mathbb{Z}: |\hat{f_j}(n-M_j)| >\alpha+\varepsilon \}
=\mu_{\hat{f_j}} (\alpha + \varepsilon).
\]
As in Lemma \ref{dilation star est}  this estimate for the distribution functions gives, $\hat{g}^*(y) \ge  \hat{f}_j^*(y) -\varepsilon$, the desired estimate for the rearrangements.
\end{proof} 

\begin{lem}\label{test fun z>1}
Let $z\ge 3$ and $A\in \mathcal{A}$. For each $\varepsilon >0$ there exists a function $f\in L^1(\mathbb{T})$ such that 
\[
f^* \le \chi_{[0,1/z)} \quad and \quad (A\omega_z)^{1/2} \le c_1 (\hat{f}^* + \varepsilon)
\]
with $c_1 = 183$.
\end{lem}
\begin{proof}
Fix $\varepsilon>0$ and let $\{(a_i,b_i)\}$ be the intervals associated with the averaging operator $A$. We will build the function $f$ in pieces and assemble them using Lemma \ref{seq fourier star est}. The first piece, $f_0=\chi_{[0,1/4z)}$, satisfies $f_0^*=\chi_{[0,1/4z)}$ so, by Lemma \ref{1/z fourier star est}, 
\[ 
\hat{f}_0^*(y) \ge (3\pi y + 9\pi(4z))^{-1}
\ge(39\pi \max(y,z))^{-1}=\omega_z(y)^{1/2}/(39\pi).
\]
If $y$ satisfies $A\omega_z(y)\le 2\omega_z(y)$, then
\begin{equation}\label{eq case 1}
A \omega_z(y)^{1/2} \le 39\sqrt{2}\pi\hat{f}_0^*(y)
\le c_1\hat{f}_0^*(y).
\end{equation}
The second piece is needed only when $z$ is contained in one of the intervals of $A$. If there is one, call it $(a_0,b_0)$. By Lemma \ref{1/z fourier star epsilon est}, with $r=\sqrt{b_0/(8z)}$ and $z$ replaced by $8z/3$, there exists a function $g_0$ such that $g_0^* = \chi_{[0,3/(8z))}$ and
\begin{align*}
\hat{g_0}^*(y)  + \varepsilon/2 
&\ge \left(3\pi y\sqrt{8z/b_0}+9\pi\left(\sqrt{b_0/(8z)}+1\right)(8z/3)\right)^{-1}\\
&=\left(6\sqrt2\pi(y/b_0)+6\sqrt2\pi+24\pi(z/b_0)^{1/2}\right)^{-1}(b_0z)^{-1/2}.
\end{align*}
If $y\in(a_0,b_0)$ then both $y/b_0$ and $z/b_0$ are less than 1 so,
\[
\hat{g_0}^*(y)  + \varepsilon/2 
\ge(12\sqrt2\pi+24\pi)^{-1}(b_0z)^{-1/2}.
\]
Also, if $y\in(a_0,b_0)$, the monotonicity of $\omega_z$ implies that,
\[
A\omega_z(y)=\frac1{b_0-a_0}\int_{a_0}^{b_0} \omega_z(t)\,dt
\le\frac1{b_0}\int_0^{b_0} \omega_z(t)\,dt
\le\frac1{b_0}\int_0^\infty\omega_z(t)\,dt=\frac2{b_0z}.
\]
Thus,
\begin{equation}\label{eq case 2}
A\omega_z(y)^{1/2}\le(12\sqrt2\pi+24\pi)\sqrt2(\hat{g_0}^*(y)  + \varepsilon/2)
\le c_1(\hat{g_0}^*(y)  + \varepsilon/2).
\end{equation}

The remaining pieces of $f$ are indexed by certain intervals of $A$. Let
\[
J = \{j\ne0: z\le a_j \le b_j /2 \}.
\]
For each $j\in J$, apply Lemma \ref{1/z fourier star epsilon est}, with $r=\sqrt{b_j/(16a_j)}$ and $z$ replaced by $16a_j/3$, to produce a function $g_j$ such that $g_j^* = \chi_{[0,3/(16a_j))}$ and
\begin{align*}
\hat{g_j}^*(y)  + \varepsilon/2 
&\ge \left(3\pi y\sqrt{16a_j/b_j}+9\pi\left(\sqrt{b_j/(16a_j)}+1\right)(16a_j/3)\right)^{-1}\\
&=\left(12\pi(y/b_j)+12\pi+24\sqrt2\pi(2a_j/b_j)^{1/2}\right)^{-1}(a_jb_j)^{-1/2}.
\end{align*} 
If $y\in(a_j,b_j)$ then both $y/b_j$ and $2a_j/b_j$ are less than 1 so,
\[
\hat{g_j}^*(y)  + \varepsilon/2 
\ge(24\pi+24\sqrt2\pi)^{-1}(a_jb_j)^{-1/2}.
\]
But for $y\in(a_j,b_j)$, $A\omega_z(y)=1/(a_jb_j)$ so
\begin{equation}\label{eq case 3}
A\omega_z(y)^{1/2}\le(24\pi+24\sqrt2\pi)(\hat{g_j}^*(y)  + \varepsilon/2)
\le c_1(\hat{g_j}^*(y)  + \varepsilon/2).
\end{equation}

To apply Lemma \ref{seq fourier star est} to the functions $f_0$, $g_0$, and $g_j$ for $j\in J$ we need to estimate the sums of the lengths of the intervals involved.
For each $j\in J$, let $m_j$ be the smallest integer such that $2^{m_j}z\le a_j$. Since $z\le a_j$, each $m_j\ge0$.  To see that $m_j\ne m_k$ for distinct $j,k\in J$, suppose $a_j\le a_k$. Since the intervals of $A$ are disjoint, $b_j\le a_k$ and we see that, $m_j>m_k$ because, $2^{m_j+1}z\le2a_j<b_j\le a_k$. Since the $m_j$ are all different,
\[
\sum_{j\in J}\frac1{a_j}\le\frac1z\sum_{j\in J}2^{-m_j}
\le\frac1z\sum_{m=0}^\infty2^{-m}=\frac2z.
\]
Therefore,
\[
\frac1{4z} + \frac3{8z} + \sum_{j \in J} \frac3{16a_j}\le \frac1z\le 1
\]
and Lemma \ref{seq fourier star est} guarantees the existence of a function $f$ such that,
\[
f^*\le \chi_{[0,1/z)},\ 
\hat{f}^* \ge  \hat{f_0}^* -\varepsilon /2,\  
\hat{f}^*\ge  \hat{g_0}^* -\varepsilon /2,\ 
\text{and }\hat{f}^* \ge  \hat{g_j}^* -\varepsilon /2\text{ for }j\in J.
\]
To see that  $(A\omega_z)^{1/2} \le c_1 (\hat{f}^* + \varepsilon)$, let $y>0$. If $A\omega_z(y)\le 2\omega_z(y)$, then (\ref{eq case 1}) shows that
\[
(A\omega_z)(y)^{1/2} \le c_1 (\hat{f_0}^*(y) + \varepsilon/2)
\le c_1 (\hat{f}^*(y) + \varepsilon).
\]
If $y$ and $z$ are in the same interval of $A$, then (\ref{eq case 2}) shows that
\[
(A\omega_z)(y)^{1/2} \le c_1 (\hat{g_0}^*(y) + \varepsilon/2)
\le c_1 (\hat{f}^*(y) + \varepsilon).
\]
Any other $y$ satisfies $A\omega_z(y)> 2\omega_z(y)$ and is not in an interval of $A$ with $z$. Since $A\omega_z(y)\ne \omega_z(y)$, $y$ is in some interval $(a_j,b_j)$ on which $\omega_z$ is not constant. Thus $z<b_j$. But $z$ is not in the interval that contains $y$ so $z\le a_j$. Therefore, 
\[
\frac1{a_jb_j}=A\omega_z(y)> 2\omega_z(y)=\frac2{y^2}\ge\frac2{b_j^2}
\]  
and we see that $a_j<b_j/2$ so $j\in J$. Now (\ref{eq case 3}) yields,
\[
(A\omega_z)(y)^{1/2} \le c_1 (\hat{g_j}^*(y) + \varepsilon/2)
\le c_1 (\hat{f}^*(y) + \varepsilon)
\]
to complete the proof.
\end{proof}

The restriction $z\ge3$ in the last lemma is a technical one and can be removed.
\begin{prop}\label{test fun all z}Let $z\ge1$ and $A\in \mathcal{A}$. For each $\varepsilon >0$ there exists a function $f\in L^1(\mathbb{T})$ such that 
\[
f^* \le \chi_{[0,1/z)} \quad and \quad (A\omega_z)^{1/2} \le c (\hat{f}^* + \varepsilon)
\]
with $c = 3 c_1 = 549$.
\end{prop}
\begin{proof}
If $z\ge 3$ then Lemma \ref{test fun z>1} implies the existence of the desired function $f$, because $c_1\le c$.

If $1\le z<3$ then we set $z=3$ in Lemma \ref{test fun z>1} to get a function $f$ such that $f^* \le \chi_{[0,1/3)}$ and $(A\omega_3)^{1/2} \le c_1 (\hat{f}^* + \varepsilon)$. Clearly,  $f^* \le \chi_{[0,1/z)}$. We also have $\omega_z \le 9\omega_3$ which implies $A\omega_z \le 9 A\omega_3$ and completes the proof.
\end{proof} 

The main result of the section follows. It uses the test functions just constructed to give a necessary condition for the Fourier series inequality (\ref{GammaLambda}). As we will see in the next section, for a large range of indices, the condition is also sufficient.
\begin{thm}\label{necessary series} Suppose $0<p<\infty$, $0<q\le \infty$, and for some $C>0$, $u,w \in L^+$ satisfy
\begin{equation}\label{GammaLambda}
\|\hat f\|_{\Lambda_q(u)}\le C\|f\|_{\Gamma_p(w)},\quad f\in L^1(\mathbb{T}).
\end{equation}
 for all $f\in L^1(\mathbb{T})$. Then
\[
\sup_{z>1}\sup_{A \in \mathcal{A}}\dfrac{\| A\omega_z \|_{q/2,u}}{\| \omega_z \|_{p/2,v}} \le c^2C^2
\]
where $c$ is the constant in Proposition \ref{test fun all z}. Here $v(t)=t^{p-2}w(1/t)$.
\end{thm}
\begin{proof} Making the change of variable $t\mapsto 1/t$ on the right-hand side, (\ref{GammaLambda}) becomes,
\begin{equation}\label{pre-gamma}
\|\hat f^*\|_{q,u} \le C \bigg( \int_0^{\infty} \bigg( \int_0^{1/t} f^* \bigg)^p \, v(t)\, dt \bigg)^{1/p}.
\end{equation}
Fix $A\in\mathcal{A}$, $z>1$, and $\varepsilon\in(0,1)$. Let $Y>0$ and use Proposition \ref{test fun all z} to choose a function $f:\mathbb{T} \to \mathbb{C}$ such that  $f^* \le \chi_{[0,1/z)}$ and 
\[
(A\omega_z)^{1/2} \le c (\hat{f}^* + (\varepsilon/c)A\omega_z(Y)^{1/2}).
\]
Since $\omega_z$ is decreasing, so is $A\omega_z$. Thus, for $y\in [0,Y)$,
\[
A\omega_z (y)^{1/2} \le  c\hat{f}^*(y) + \varepsilon A\omega_z(Y)^{1/2} \le  c\hat{f}^*(y) + \varepsilon A\omega_z(y)^{1/2},
\]
so $(1-\varepsilon)A\omega_z (y)^{1/2} \le c\hat{f}^*(y)$. Therefore,
\[
(1-\varepsilon)^2\|(A\omega_z)\chi_{[0,Y)}\|_{q/2,u}
\le c^2\|\hat f^*\chi_{[0,Y)}\|_{q,u}^2
\le c^2\|\hat f^*\|_{q,u}^2.
\]
But, for all $y$, 
\[
\int_0^{1/y}f^*(t)\,dt\le\int_0^{1/y}\chi_{[0,1/z)}(t)\,dt=\omega_z(y)^{1/2},
\]
so, using (\ref{pre-gamma}),
\[
\|\hat f^*\|_{q,u}
\le C\|\omega_z^{1/2}\|_{p,v}
=C\|\omega_z\|_{p/2,v}^{1/2}.
\]
We conclude that
\[
(1-\varepsilon)^2\|(A\omega_z)\chi_{[0,Y)}\|_{q/2,u}\le c^2C^2\|\omega_z\|_{p/2,v}.
\]
Letting $Y\to\infty$, and then $\varepsilon\to0$, gives, 
\[
\|A\omega_z\|_{q/2,u}\le c^2C^2\|\omega_z\|_{p/2,v}
\]
and completes the proof. 
\end{proof}
A slight simplification of the above proof gives the corresponding result for the Fourier series inequality between $\Lambda$-spaces. Notice that in this case both the Fourier inequality and the weight condition depend only on the values of $w(t)$ for $0<t<1$.
\begin{cor}\label{weak necessary series} Suppose $0<p<\infty$, $0<q< \infty$, and for some $C>0$, $u,w \in L^+$ satisfy
\[
\|\hat f\|_{\Lambda_q(u)}\le C\|f\|_{\Lambda_p(w)},\quad f\in L^1(\mathbb{T}).
\] for all $f\in L^1(\mathbb{T})$. Then
\[
\sup_{z>1}\sup_{A \in \mathcal{A}}\dfrac{\| A\omega_z \|_{q/2,u}}{\| \chi_{(0,1/z)} \|_{p/2,w}} \le c^2C^2
\]
where $c$ is the constant in Proposition \ref{test fun all z}. 
\end{cor}

\section{Main Results}\label{WC}

In this section we present weight conditions that ensure the boundedness of the Fourier coefficient map between Lorentz spaces. For a large range of indices these coincide with the necessary conditions obtained in the previous section to give a characterization of exactly those weights for which the map is bounded. The focus is on the inequality (\ref{GammaGamma}), which expresses the boundedness of the Fourier coefficient map from $\Gamma_p(w)$ to $\Gamma_q(u)$ but we will see that exactly the same weight conditions give boundedness from $\Gamma_p(w)$ to $\Lambda_q(u)$. Under mild conditions on $w$ the boundedness from $\Lambda_p(w)$ to $\Lambda_q(u)$ is also equivalent.

While the most interesting results involve weights $w$ that are supported on $[0,1]$, other weights are permitted. The interested reader may verify that both the Fourier inequalities and the various weight conditions depend on $w\chi_{(1,\infty)}$ only through the value of $\int_1^\infty w(t)\,\frac{dt}{t^p}$. Similarly, any weight $u$ is permitted, but the most interesting cases involve weights $u$ that are constant on $[n-1,n)$ for $n=1,2\dots$. See Theorem \ref{necessary sufficient averaging omega_z series}\ref{NS3} for an indication that only the values $\int_{n-1}^nu(t)\,dt$ for $n=1,2,\dots$, are of significance.

For sufficiency of the weight conditions we actually prove the boundedness result for a large class of operators that includes the Fourier coefficient map. Let $(X,\mu)$ and $(Y,\nu)$ be $\sigma$-finite measure spaces with $\mu(X)=1$, and let $T$ be a sublinear operator from $L^1_\mu+L^2_\mu$ to $L^2_\nu+L^\infty_\nu$. We say that  $T$ is of type $(1,\infty)$ and $(2,2)$ provided $T$ is a bounded map both from $L^1_\mu$ to $L^{\infty}_\nu$ and from $L^2_\mu$ to $L^2_\nu$. The Fourier coefficient map is one such operator; in this case $\mu$ is Lebesgue measure on $[0,1]$ and $\nu$ is counting measure on $\mathbb{Z}$. 

Let $\mathcal{T}$ denote the collection of all sublinear operators $T$ of type $(1,\infty)$ and $(2,2)$, over all probability measures $\mu$ and $\sigma$-finite measures $\nu$. Propositions \ref{sufficient series} and \ref{CC}, below, give several weight conditions that are sufficient for the inequality,
\begin{equation}\label{general norm form}
\|Tf\|_{\Gamma_q(u)} \le C \|f\|_{\Gamma_p(w)},\quad f\in L^1_\mu,
\end{equation} 
to hold for all $T\in\mathcal{T}$. Recalling that $v(t)=t^{p-2}w(1/t)$ it is easy to rewrite this as,
\begin{equation}\label{integral form}
\bigg( \int_0^{\infty} (Tf)^{**}(t)^q u(t) dt \bigg)^{1/q} \le C \bigg( \int_0^{\infty} \bigg( \int_0^{\frac{1}{t}} f^*(s)\,ds \bigg)^p v(t) dt \bigg)^{1/p},
\quad f\in L^1_\mu.
\end{equation}
(Note that since $\mu$ is a finite measure, $L^1_\mu+L^2_\mu=L^1_\mu$.)

The results of this section are based on the following corollary of a rearrangement estimate from \cite{JT}. 
\begin{prop}\label{Jodeit-Torchinsky} Suppose $(X,\mu)$ and $(Y,\nu)$ are $\sigma$-finite measure spaces and let $T$ be a sublinear operator from $L^1_\mu+L^2_\mu$ to $L^2_\nu+L^\infty_\nu$. Then $T\in\mathcal{T}$ if and only if there exists a constant $D_T$ such that
\begin{equation}\label{JTplus}
\int_0^z (Tf)^{**}(t)^2 \ dt \le D_T \int_0^z \bigg( \int_0^{\frac{1}{t}} f^*(s)\,ds \bigg)^2 \ dt 
\end{equation}
for all $z>0$ and $f\in L^1_\mu+L^2_\mu$.
\end{prop}

This result appears in \cite{JT} with $(Tf)^*$ instead of the larger $(Tf)^{**}$. But Hardy's inequality shows if $T\in \mathcal T$, then so is the map $f\mapsto (Tf)^{**}$. Since $((TF)^{**})^*=(TF)^{**}$, we obtain the statement above.
In the case that $T$ is the Fourier coefficient map, we may take $D_T=8$.

The next two theorems give sufficient conditions for the Fourier inequality (\ref{GammaGamma}). Recall that a function $h:(0,\infty)\to[0,\infty)$ is in $P\cap\Omega_{2,0}$ provided $t^2h(t)$ is increasing, $h(t)$ is decreasing, and  $h(t)$ is constant on $(0,1)$.
\begin{prop}\label{sufficient series} Suppose $0<p<\infty$, $0<q<\infty$, and $u,v\in L^+$. Let
\[
C_\Theta=\sup_{h\in P\cap\Omega_{2,0}}\frac{\|h\|_{\Theta_{q/2}(u)}}{\|h\|_{p/2,v}}.
\]
Then for each $T\in\mathcal{T}$ the inequality (\ref{general norm form}) holds with $C=\sqrt{D_TC_\Theta }$. In particular, the Fourier inequalities (\ref{GammaGamma}) and (\ref{GammaLambda}) hold with $C=\sqrt{8C_\Theta}$. Here $w(t)=t^{p-2}v(1/t)$.
\end{prop}
\begin{proof} Let $T\in\mathcal{T}$ and let $(X,\mu)$ be its associated probability space. Fix $f\in L^1_\mu$. Let $h_f$ and $\varphi_f$ be defined by
\[
h_f(t) = \bigg(\int_0^{1/t} f^*(s)\,ds \bigg)^2  \quad \text{and}\quad   \varphi_f(t) = (1/D_T)(Tf)^{**}(t)^2,
\]
where $D_T$ is the constant from Proposition \ref{Jodeit-Torchinsky}.
Notice that $h_f(t)$ is decreasing and $t^2 h_f(t) = f^{**}(1/t)^2$ is increasing. Also, since $\mu(X)=1$, $f^*$ vanishes outside the interval $(0,1)$ and therefore $h_f$ is constant on $(0,1)$. It follows that $h_f\in P\cap\Omega_{2,0}$. In addition, $\varphi_f$ is decreasing and Proposition \ref{Jodeit-Torchinsky} implies that $\varphi_f^{**}\le h_f^{**}$. 
Therefore,
\[
\| \varphi_f \|_{q/2,u}\le\| h_f \|_{\Theta_{q/2}(u)}\le C_\Theta\| h_f \|_{p/2,v}.
\]
This implies (\ref{integral form}) with $C=\sqrt{D_TC_\Theta}$. 

Since $w(t)=t^{p-2}v(1/t)$, (\ref{integral form}) becomes (\ref{general norm form}). Taking $T$ to be the Fourier coefficient map, and $D_T=8$, we obtain (\ref{GammaGamma}) with $C=\sqrt{8C_\Theta}$. The weaker inequality (\ref{GammaLambda}) is an immediate consequence. 
\end{proof}
\begin{rem}\label{example} Suppose $q=p$ and $u=v\in B_{p/2}$. Then $\Theta_{p/2}(u)=\Lambda_{p/2}(u)$, with equivalent norms, so for $h$ decreasing, $\|h\|_{\Theta_{p/2}(u)}\approx\|h\|_{\Lambda_{p/2}(u)}=\|h\|_{p/2,v}$. It follows that $C_\Theta<\infty$ and we have, with $w(t)=t^{p-2}u(1/t)$,
\[
\|\hat f\|_{\Gamma_p(u)}\le C\|f\|_{\Gamma_p(w)},\quad f\in L^1.
\]
In particular, when $1<p<2$ and $u(t)=v(t)=t^{p-2}$, we recover the well-known fact that the Fourier transform maps $L^p$ into the power-weighted Lorentz space $\ell^{p',p}$. Recall that,  
\[
\|\hat f\|_{\ell^{p',p}}=\left(\int_0^\infty t^{p/p'-1}(\hat f)^*(t)^p\,dt\right)^{1/p}.
\]
\end{rem}
Despite this example, $C_\Theta$ can often be difficult to estimate directly, so we provide a number of estimates in the next proposition. One that will figure prominently in our weight characterization is,
\begin{equation}\label{Comegadefn}
C_\omega=\sup_{z>1}\frac{\|\omega_z\|_{\Theta_{q/2}(u)}}{\|\omega_z\|_{p/2,v}}.
\end{equation}
Recall that $\omega_z(t)=\min(t^{-2},z^{-2})$. Also recall that $u^o$ denotes the level function of $u$ with respect to Lebesgue measure. Both will be needed in the statement and proof of the next theorem.
In view of Proposition \ref{sufficient series}, each of the following upper bounds for $C_\Theta$ gives a sufficient condition for (\ref{general norm form}) and hence for (\ref{GammaGamma}).

\begin{prop}\label{CC} Suppose $0<p<\infty$, $0<q<\infty$, and $u,v\in L^+$.
\begin{enumerate}
\item\label{TG} For any $p$ and $q$,
\[
C_\Theta\le \sup_{h\in P\cap\Omega_{2,0}}\bigg(\int_0^\infty\bigg(\frac1t\int_0^th(s)\,ds\bigg)^{q/2}u(t)\,dt\bigg)^{2/q}
\bigg(\int_0^\infty h(t)^{p/2}v(t)\,dt\bigg)^{-2/p}.
\]
\item\label{TA} If $q\ge2$, then 
\begin{align*}
C_\Theta=&\sup_{h\in P\cap\Omega_{2,0}}\sup_{A\in \mathcal{A}}\frac{\|Ah\|_{q/2,u}}{\|h\|_{p/2,v}}\quad\text{and}\\
C_\Theta\le& \sup_{h\in P\cap\Omega_{2,0}}\bigg(\int_0^\infty h(t)^{q/2}u^o(t)\,dt\bigg)^{2/q}
\bigg(\int_0^\infty h(t)^{p/2}v(t)\,dt\bigg)^{-2/p}.
\end{align*}
\item\label{Tu} If $p\le q$ and $q\ge2$  then 
\[
C_\Theta\le(4q')^{2/q}\sup_{1<1/x<y}\bigg(\frac 1{xy}\int_0^yu(t)\,dt\bigg)^{2/q}
\bigg(\int_0^\infty\omega_x(t)^{p/2}w(t)\,dt\bigg)^{-2/p}.
\]
\item\label{TO} If $p\le2\le q$ then $C_\Theta\le 2C_\omega$.

\end{enumerate}
\end{prop}
\begin{proof} Inequality (\ref{LTG}) shows that for any decreasing $h$,
\[
\|h\|_{\Theta_{q/2}(u)}\le\|h\|_{\Gamma_{q/2}(u)}
=\bigg(\int_0^\infty\bigg(\frac1t\int_0^th(s)\,ds\bigg)^{q/2}u(t)\,dt\bigg)^{2/q},
\]
which proves \ref{TG}. 

When $q\ge2$, Corollary 2.4 of \cite{FTLS} shows that for each decreasing $h$,
\begin{equation}\label{Cor24}
\|h\|_{\Theta_{q/2}(u)}=\sup_{A\in\mathcal{A}}\|Ah\|_{q/2,u}\le\|h\|_{q/2,u^o}.
\end{equation}
This gives both the equation and the upper bound in \ref{TA}. It also shows that when $q\ge2$, 
\begin{equation}\label{Comega}
C_\omega=\sup_{z>1}\sup_{A\in \mathcal{A}}\frac{\|A\omega_z\|_{q/2,u}}{\|\omega_z\|_{p/2,v}},
\end{equation}
which will be useful later.

For \ref{Tu} we begin by applying Corollary \ref{maligranda series fourier} to the upper bound from \ref{TA}. Since $p\le q$,
\[
C_\Theta\le\sup_{h\in P\cap\Omega_{2,0}}\frac{\|h\|_{q/2,u^o}}{\|h\|_{p/2,v}}\le 2^{2/q}\sup_{z>1}\frac{\|\omega_z\|_{q/2,u^o}}{\|\omega_z\|_{p/2,v}}.
\]
Let $x=1/z$, and make the change of variable $t\mapsto 1/t$ in the denominator to get,
\[
\|\omega_z\|_{p/2,v}
=\bigg(\int_0^\infty\min(t^p,x^p)t^{-p}w(t)\,dt\bigg)^{2/p}
=x^2\bigg(\int_0^\infty\omega_x(t)^{p/2}w(t)\,dt\bigg)^{2/p}.
\]
To estimate the numerator, observe that since $u^o$ is decreasing,
\[
\int_z^\infty u^o(t)\,\frac{dt}{t^q}\le u^o(z)\int_z^\infty \,\frac{dt}{t^q}
=u^o(z)\frac{z^{-q}}{q-1}\int_0^z\,dt
\le \frac{z^{-q}}{q-1}\int_0^zu^o(t)\,dt.
\]
Thus, 
\[
\|\omega_z\|_{q/2,u^o}= \bigg(z^{-q}\int_0^z u^o(t)\,dt+\int_z^\infty u^o(t)\,\frac{dt}{t^q}\bigg)^{2/q}
\le\bigg(q'z^{-q}\int_0^zu^o(t)\,dt\bigg)^{2/q}.
\]
But Lemma 2.5 of \cite{FTLS} shows that
\[
\frac1z\int_0^zu^o(t)\,dt\le2\sup_{y\ge z}\frac1y\int_0^yu(t)\,dt,
\]
so,
\[
\|\omega_z\|_{q/2,u^o}
\le\bigg(2q'z^{1-q}\sup_{y\ge z}\frac 1y\int_0^yu(t)\,dt\bigg)^{2/q}
= x^2\bigg(2q'\sup_{1/x\le y}\frac 1{xy}\int_0^yu(t)\,dt\bigg)^{2/q}.
\]
This estimate gives,
\[
C_\Theta\le2^{2/q}\sup_{1/x>1}x^2\bigg(2q'\sup_{1/x\le y}\frac 1{xy}\int_0^yu(t)\,dt\bigg)^{2/q}x^{-2}
\bigg(\int_0^\infty\omega_x(t)^{p/2}w(t)\,dt\bigg)^{-2/p},
\]
which simplifies to the conclusion of \ref{Tu}.

To prove \ref{TO}, apply Corollary \ref{averaging operator norm series fourier} to the equation from \ref{TA} to get
\[
C_\Theta=\sup_{h\in P\cap\Omega_{2,0}}\sup_{A\in \mathcal{A}}\frac{\|Ah\|_{q/2,u}}{\|h\|_{p/2,v}}
\le2\sup_{z>1}\sup_{A\in \mathcal{A}}\frac{\|A\omega_z\|_{q/2,u}}{\|\omega_z\|_{p/2,v}}
=2C_\omega,
\]
where the last equality is (\ref{Comega}).
\end{proof}

Next we combine the sufficiency results above with the necessary conditions obtained in Section \ref{NC} to obtain a necessary and sufficient condition for the boundedness of the Fourier coefficient map between weighted Lorentz spaces. Recall that $v(t)=t^{p-2}w(1/t)$ and refer to expressions (\ref{Comegadefn}) and (\ref{Comega}) for the constant $C_\omega$.
\begin{thm}\label{NandS} Let $0<p\le 2\le q<\infty$ and $u,w \in L^+$. The Fourier inequality, 
\[
\|\hat f\|_{\Gamma_q(u)}\le C\|f\|_{\Gamma_p(w)},\quad f\in L^1(\mathbb{T}), 
\]
holds if and only if $C_\omega<\infty$. Moreover, for the best constant $C$,
\[
\frac{\sqrt C_\omega}{549}\le C\le 4\sqrt{C_\omega}.
\] 
\end{thm}
\begin{proof} Let $C$ be the least constant, finite or infinite, in the above Fourier inequality. Propositions \ref{sufficient series} and \ref{CC}\ref{TO} show that for any $f\in L^1(\mathbb{T})$, 
\[
\|\hat f\|_{\Gamma_q(u)}
\le\sqrt{8C_\Theta}\|f\|_{\Gamma_p(w)}
\le4\sqrt{C_\omega}\|f\|_{\Gamma_p(w)}.
\]
Thus, $C\le4\sqrt{C_\omega}$.

On the other hand, inequality (\ref{LTG}) shows that we also have,
\[
\|\hat f\|_{\Lambda_q(u)}\le C\|f\|_{\Gamma_p(w)},\quad f\in L^1(\mathbb{T}), 
\]
so (\ref{Comega}) and Theorem \ref{necessary series} give,
\[
C_\omega=\sup_{z>1}\sup_{A\in \mathcal{A}}\frac{\|A\omega_z\|_{q/2,u}}{\|\omega_z\|_{p/2,v}}\le(549C)^2.
\]
This completes the proof.
\end{proof}
In the case $q=2$ the necessary and sufficient condition $C_\omega<\infty$ can be put in a form that is especially simple to estimate.
\begin{thm}\label{NSq=2} Let $0<p\le 2$ and $u,w \in L^+$. If $C$ is the best constant in the Fourier inequality, 
\[
\|\hat f\|_{\Gamma_2(u)}\le C\|f\|_{\Gamma_p(w)},\quad f\in L^1(\mathbb{T}), 
\]
then
\[
C_{xy}/549\le C\le 8C_{xy},
\] 
where
\[
C_{xy}=\sup_{1<1/x<y}\bigg(\frac 1{xy}\int_0^yu(t)\,dt\bigg)^{1/2}
\bigg(x^{-p}\int_0^xw(t)\,dt+\int_x^\infty w(t)\,\frac{dt}{t^p}\bigg)^{-1/p}.
\]
\end{thm}
\begin{proof} With $q=2$, Proposition \ref{CC}\ref{Tu} shows $C_\Theta\le 8C_{xy}^2$ giving $C\le \sqrt{8C_\Theta}\le 8C_{xy}$.

It remains to show that $C_{xy}/549\le C$. Since $\frac1y\int_0^yu^o(t)\,dt$ is a decreasing function that majorizes $\frac1y\int_0^yu(t)\,dt$, we have 
\begin{equation}\label{levelbigger}
\sup_{1/x<y}\frac1y\int_0^yu(t)\,dt\le x\int_0^{1/x}u^o(t)\,dt.
\end{equation}
So, taking $z=1/x$ and applying Lemma 2.2 of \cite{FTLS},
\[
\sup_{1/x<y}\frac1{xy}\int_0^yu(t)\,dt
\le \int_0^z u^o(t)\,dt\le z^2\|\omega_z\|_{1,u^o}
=z^2\sup_{A\in\mathcal A}\|A\omega_z\|_{1,u}.
\]
Also,
\[
x^{-p}\int_0^xw(t)\,dt+\int_x^\infty w(t)\,\frac{dt}{t^p}
=z^p\|\omega_z\|_{p/2,v}^{p/2}.
\]
Therefore, by (\ref{Comega}) and Theorem \ref{NandS},
\[
C_{xy}^2\le\sup_{z>1}\sup_{A\in\mathcal A}\frac{\|A\omega_z\|_{1,u}}{\|\omega_z\|_{p/2,v}}=C_\omega\le(549C)^2.
\]
Taking square roots completes the proof.
\end{proof}
Examining the proofs of the last two theorems gives a more general result. We record it without tracking the estimates of the constants involved.

\begin{thm}\label{necessary sufficient averaging omega_z series} Let $0<p\le 2\le q<\infty$ and $u,w \in L^+$. The following are equivalent.
\begin{enumerate}
\item\label{NS1} For each $T\in \mathcal{T}$ there exists a finite constant $C$ such that 
\begin{equation}\label{NST}
\|Tf\|_{\Gamma_q(u)} \le C \|f\|_{\Gamma_p(w)},\quad f\in L^1_\mu.\text{ (Here $\mu$ depends on $T$.)}
\end{equation}
\item\label{NS2} There exists a finite constant $C$ such that 
\[
\|\hat f\|_{\Gamma_q(u)} \le C \|f\|_{\Gamma_p(w)},\quad f\in L^1(\mathbb{T}).
\]
\item\label{NS3} There exists a finite constant $C$ such that 
\[
\|\hat f\|_{\Lambda_q(u)} \le C \|f\|_{\Gamma_p(w)},\quad f\in L^1(\mathbb{T}).
\]
\item\label{NS4} $C_\omega<\infty$.
\end{enumerate}
When $q=2$, $C_{xy}<\infty$ is also equivalent.
\end{thm}
\begin{proof}
Since the Fourier coefficient map is in $\mathcal{T}$, \ref{NS1} implies \ref{NS2}. Inequality (\ref{LTG}) shows that \ref{NS2} implies \ref{NS3}. But if \ref{NS3} holds for some finite constant $C$, then the proof of Theorem \ref{NandS} provides $C_\omega\le (549 C)^2<\infty$ and gives \ref{NS4}. To complete the circle apply Propositions \ref{sufficient series} and \ref{CC}\ref{TO} to see that for each $T\in \mathcal{T}$, (\ref{NST}) holds with  $C=\sqrt{D_TC_\Theta}\le\sqrt{2D_TC_\omega}<\infty$.

The last statement of the theorem follows from Theorem \ref{NSq=2}.
\end{proof} 

Under a mild a priori condition on the weight $u$, the necessary and sufficient condition simplifies and the above theorem extends to a larger range of indices. Recall that if $u$ is decreasing, then $u\in B_{1,\infty}$ and $u\in B_{q/2}$ for every $q>2$.
\begin{prop}\label{NSunice} Let $0<p\le q<\infty$, $2\le q$ and $u,w \in L^+$. Suppose that $u\in B_{q/2}$ if $q>2$, and $u\in B_{1,\infty}$ if $q=2$. Then the Fourier inequality, 
\[
\|\hat f\|_{\Gamma_q(u)} \le C \|f\|_{\Gamma_p(w)},\quad f\in L^1(\mathbb{T}).
\]
holds if and only if 
\begin{equation}\label{nolevel}
\sup_{0<x<1}\bigg(\int_0^{1/x}u(t)\,dt\bigg)^{1/q}
\bigg(x^{-p}\int_0^xw(t)\,dt+\int_x^\infty w(t)\,\frac{dt}{t^p}\bigg)^{-1/p}
\end{equation}
is finite. In fact, under the above hypothesis on $u$, statements \ref{NS1}, \ref{NS2} and \ref{NS3} of Theorem \ref{necessary sufficient averaging omega_z series} are all equivalent to the finiteness of (\ref{nolevel}).
\end{prop}
\begin{proof} First suppose $q=2$. The $B_{1,\infty}$ condition on $u$ shows that there exists a constant $b_1(u)$ such that for each $x$,
\[
\int_0^{1/x}u(t)\,dt\le \sup_{1<xy}\frac1{xy}\int_0^yu(t)\,dt\le b_1(u)\int_0^{1/x}u(t)\,dt
\]
so $C_{xy}<\infty$ if and only if (\ref{nolevel}) is finite. Since $p\le q=2$, Theorem \ref{necessary sufficient averaging omega_z series} completes the proof.

Now consider the case $q>2$ and suppose that (\ref{nolevel}) is finite. Since $u\in B_{q/2}$ there exists a constant $c$ such that inequality (\ref{AMH}) holds. This and (\ref{LTG}) give, for any decreasing $h$,
\[
\|h\|_{\Theta_{q/2}(u)}\le\|h\|_{\Gamma_{q/2}(u)}\le c\|h\|_{\Lambda_{q/2}(u)}=c\|h\|_{q/2,u}.
\]
Therefore
\[
C_\Theta\le c\sup_{h\in P\cap\Omega_{2,0}}\frac{\|h\|_{q/2,u}}{\|h\|_{p/2,v}}
\le c2^{2/q}\sup_{z>1}\frac{\|\omega_z\|_{q/2,u}}{\|\omega_z\|_{p/2,v}},
\]
where the second inequality is from Corollary \ref{maligranda series fourier}.

But
\[
\int_z^\infty u(t)\,\frac{dt}{t^q}\le\frac1{z^{q/2}}\int_z^\infty u(t)\,\frac{dt}{t^{q/2}}
\le\frac {b_{q/2}(u)}{z^q}\int_0^zu(t)\,dt,
\]
so
\[
\|\omega_z\|_{q/2,u}
=\bigg(\frac1{z^q}\int_0^zu(t)\,dt+\int_z^\infty u(t)\,\frac{dt}{t^q}\bigg)^{q/2}
\le(1+b_{q/2}(u))^{2/q}\bigg( \frac1{z^q}\int_0^zu(t)\,dt\bigg)^{2/q}.
\]
Therefore, $C_\Theta$ is bounded above by a multiple of,
\[
\sup_{z>1}{\bigg(\frac 1{z^q}\int_0^zu(t)\,dt\bigg)^{2/q}}
{\bigg(\frac 1{z^p}\int_0^zv(t)\,dt+\int_z^\infty v(t)\,\frac{dt}{t^p}\bigg)^{-2/p}}.
\]
Using $w(t)=t^{p-2}v(1/t)$ and letting $z=1/x$ we see that the last expression is equal to the square of (\ref{nolevel}) and therefore $C_\Theta<\infty$. Proposition \ref{sufficient series} shows that Part \ref{NS1} of Theorem \ref{necessary sufficient averaging omega_z series} holds. For any $p$ and $q$, Part \ref{NS1} implies Part \ref{NS2}.

For the converse, still in the case $q>2$, suppose the Fourier inequality, 
\[
\|\hat f\|_{\Gamma_q(u)} \le C \|f\|_{\Gamma_p(w)},\quad f\in L^1(\mathbb{T}).
\]
holds. That is, Part \ref{NS2} of Theorem \ref{necessary sufficient averaging omega_z series} holds.  Then (\ref{LTG}) shows that Part \ref{NS3} also holds, so we may apply Theorem \ref{necessary series}. Since the identity operator is in $\mathcal A$, we have,
\[
\sup_{z>1}\frac{\|z^{-2}\chi_{(0,z)}\|_{q/2,u}}{\|\omega_z\|_{p/2,v}}
\le\sup_{z>1}\frac{\|\omega_z\|_{q/2,u}}{\|\omega_z\|_{p/2,v}}
\le\sup_{z>1}\sup_{A\in\mathcal A}\frac{\|A\omega_z\|_{q/2,u}}{\|\omega_z\|_{p/2,v}}<\infty.
\]
Using $w(t)=t^{p-2}v(1/t)$ and letting $z=1/x$, we see that (\ref{nolevel}) is finite. This completes the proof.
\end{proof}

One consequence of this theorem is an analogue for the Fourier coefficient map of Theorem 2 in \cite{BH2}, a result for the Fourier transform. In the original result, $u$ was assumed to be decreasing. In this analogue we have weakened this condition; a decreasing function is in both $B_{q/2}$ and $B_{1,\infty}$.
\begin{thm}\label{heinig lorentz sufficient series} Let $u$ and $w$ be weight functions on $(0,\infty)$.
\begin{enumerate}
\item\label{BHA1} Suppose $1<p\le q<\infty$, $q\ge 2$, and $w\in B_p$. Also suppose that $u\in B_{q/2}$ if $q>2$, and $u\in B_{1,\infty}$ if $q=2$. If
\begin{equation}\label{BHC}
\sup_{0<x<1}x\bigg(\int_0^{1/x} u(t)\,dt \bigg)^{1/q} \bigg( \int_0^x w(t)\,dt \bigg)^{-1/p} < \infty
\end{equation}
then there is a $C > 0$ such that 
\begin{equation}\label{BHA}
\|\hat f\|_{\Lambda_q(u)} \le C \|f\|_{\Lambda_p(w)},\quad f\in L^1(\mathbb{T}).
\end{equation}
\item\label{BHA2} Conversely, if (\ref{BHA}) is satisfied for any weight functions $u$ and $w$ on $(0,\infty)$ and for 
$1<p,q<\infty$, then (\ref{BHC}) holds. 
\end{enumerate}
\end{thm}
\begin{proof} As always, $v(t) =t^{p-2}w(1/t)$.
For Part \ref{BHA1}, since $w\in B_p$, $\|\cdot\|_{\Lambda_p(w)}$ and $\|\cdot\|_{\Gamma_p(w)}$ are equivalent norms. Therefore (\ref{BHA}) is equivalent to Theorem \ref{necessary sufficient averaging omega_z series}\ref{NS3}. A routine calculation using $w\in B_p$ shows that (\ref{BHC}) is equivalent to the finiteness of (\ref{nolevel}). Now Proposition \ref{NSunice} completes the proof of Part \ref{BHA1}. 

For part \ref{BHA2}, apply Corollary \ref{weak necessary series}, taking $A$ to be the identity operator, to get,
\[
\sup_{z>1}\frac{\|\omega_z\|_{q/2,u}}{\|\chi_{(0,1/z)}\|_{p/2,w}}<\infty.
\]
Since $z^{-2}\chi_{(0,z)}\le\omega_z$, this implies
\[
\sup_{z>1}\left(z^{-q}\int_0^zu(t)\,dt\right)^{q/2}
\left(\int_0^{1/z}w(t)\,dt\right)^{-2/p}<\infty.
\]
Replacing $z$ by $1/x$ and taking square roots proves (\ref{BHC}).
\end{proof}

As a final result for this section we show the sufficiency, for the Fourier coefficient map on Lorentz spaces, of a weight condition analogous to one used in Theorem 1 of \cite{BH2}, a result for the Fourier transform on Lebesgue spaces. 

\begin{thm} Let $1<p\le q <\infty$ and $2\le q$. Assume $u$ and $w$ are weight functions on $(0,\infty)$. If 
\[
\sup_{0<x<1} \ \bigg(\int_0^{1/x} u^o(t)\,dt \bigg)^{1/q} \bigg( \int_0^x w(t)^{1-p'}\,dt \bigg)^{1/p'} < \infty
\]
then there exists $C>0$ such that
\[
\|\hat{f}\|_{\Lambda_q(u)} \le C \|f\|_{\Lambda_p(w)}
\]
for all $f\in L^1(\mathbb{T})$.
\end{thm}
\begin{proof} Let $\sigma(t)=t^{q-2}u^o(1/t)$ so that 
\[\int_0^{1/x} u^o(t)\,dt=\int_x^\infty \sigma(t)\,\frac{dt}{t^q}
\] 
and hence, 
\[
\sup_{0<x<1}\bigg(\int_0^{1/x}u^o(t)\,dt\bigg)^{1/q}\bigg(x^{-q}\int_0^x\sigma(t)\,dt+\int_x^\infty \sigma(t)\,\frac{dt}{t^q}\bigg)^{-1/q}\le1.
\]
In view of (\ref{levelbigger}) we may apply Propositions \ref{CC}\ref{Tu} and \ref{sufficient series}, in the case $p=q$ and $w=\sigma$, to see that 
\[
\|\hat f\|_{\Gamma_q(u)}\le \bar C\|f\|_{\Gamma_q(\sigma)}, \quad f\in L^1(\mathbb{T}),
\]
for some $\bar C<\infty$.

Let $B$ be the supremum in the hypothesis of the theorem. Then,
\[
\sup_{0<x<1} \bigg(\int_x^1 \sigma(t)\,\frac{dt}{t^q} \bigg)^{1/q} \bigg( \int_0^x w(t)^{1-p'}\,dt \bigg)^{1/p'}\le B<\infty.
\]
By Theorem 1 in \cite{B}, there exists a finite constant $c$ such that the weighted Hardy inequality,
\[
\bigg(\int_0^1 \bigg(\frac1t\int_0^tg(s)\,ds\bigg)^q \sigma(t)\,dt\bigg)^{1/q}
\le c\bigg(\int_0^1 g(t)^pw(t)\,dt\bigg)^{1/p},
\]
holds for all $g\ge0$. 

If $f\in L^1(\mathbb{T})$ then $f^*$ is zero on $(1,\infty)$. Therefore,
\[
\|f\|_{\Gamma_q(\sigma)}^q=\int_0^1\bigg(\frac1t\int_0^tf^*(s)\,ds\bigg)^q\sigma(t)\,dt
+\bigg(\int_0^1 f^*(s)\,ds\bigg)^q\int_1^\infty\sigma(t)\,\frac{dt}{t^q}.
\]
The weighted Hardy inequality above, with $g=f^*$, shows that the first term is bounded above by,
\[
c^q\bigg(\int_0^1 f^*(t)^pw(t)\,dt\bigg)^{q/p}=c^q\|f\|_{\Lambda_p(w)}^q.
\]
H\"older's inequality shows that the second term is bounded above by
\[
\bigg(\int_0^1 f^*(t)^pw(t)\,dt\bigg)^{q/p}\bigg(\int_0^1 w(t)^{1-p'}\,dt\bigg)^{q/p'}\int_1^\infty\sigma(t)\,\frac{dt}{t^q}
\le B^q\|f\|_{\Lambda_p(w)}^q.
\]
Putting these together, we have,
\[
\|\hat f\|_{\Lambda_q(u)}\le \|\hat f\|_{\Gamma_q(u)}\le \bar C\|f\|_{\Gamma_q(\sigma)}\le (c^q+B^q)^{1/q}\bar C\|f\|_{\Lambda_q(w)}.
\]
This completes the proof.
\end{proof}

\section{Applications}\label{App}

The space $L\log L$ consists of functions $f$ defined on $[0,1]$ for which the integral of $|f|\log (2+|f|)$ is finite. It is well known that the space coincides with the Lorentz space $\Lambda_1(1-\log(t))=\Gamma_1(1)$. See, for example, Corollary 10.2 in \cite{BR}. 

Taking $p=1$ and $w=1$ in Proposition \ref{CC}\ref{Tu} and Theorem \ref{NSq=2} gives a description of spaces that contain the Fourier series of all functions in $L\log L$.
\begin{thm} Suppose $2\le q<\infty$ and $u$ is a weight. If
\begin{equation}\label{LlogLCond}
\sup_{z>1}\frac z{(1+\log z)^q}\sup_{y>z}\frac1y\int_0^y u(t)\,dt<\infty
\end{equation}
then $\mathcal F:L\log L\to\Gamma_q(u)$ and hence $\mathcal F:L\log L\to\Lambda_q(u)$. When $q=2$, condition (\ref{LlogLCond}) is also necessary.
\end{thm}

\begin{rem} If $u$ is decreasing, or satisfies the weaker condition $u\in B_{1,\infty}$ then $\sup_{y>z}\frac1y\int_0^y u(t)\,dt$ may be replaced by $\frac1z\int_0^z u(t)\,dt$ in the previous theorem.
\end{rem}
\def\L{L}\def\l{l}\def\k{(1+|\log t|)}

But $L\log L$ is only one of a large class of Lorentz spaces known as Lorentz-Zygmund spaces. One can define Lorentz-Zygmund spaces for functions on any $\sigma$-finite measure space by letting,
\[
\|f\|_{\mathcal L^{r,p}(\log\mathcal L)^\alpha}
=\begin{cases}\left(\int_0^\infty\big[t^{1/r}\k^\alpha f^*(t)\big]^p\,\frac{dt}t\right)^{1/p},
&0<p<\infty,\\
\sup_{0<t<\infty}t^{1/r}\k^\alpha f^*(t),&p=\infty,\end{cases}
\]
and setting $\mathcal L^{r,p}(\log\mathcal L)^\alpha=\{f:\|f\|_{\mathcal L^{r,p}(\log\mathcal L)^\alpha}<\infty\}$. When the underlying measure is Lebesgue measure on $[0,1]$ we write $L^{r,p}(\log L)^\alpha$ and when the underlying measure is counting measure on $\Bbb Z$ we write $\ell^{r,p}(\log\ell)^\alpha$.

The Lorentz-Zygmund spaces were introduced and studied in \cite{BR}. They are special cases of the Lorentz $\Lambda_p(w)$-spaces. Specifically, $L^{r,p}(\log L)^\alpha=\Lambda_p(w)$, where 
\[
w(t)=t^{p/r-1}(1-\log t)^{\alpha p}\chi_{(0,1)}(t),
\]
and $\ell^{r,p}(\log\ell)^\alpha=\Lambda_p(u)$, where 
\[
u(t)=t^{p/r-1}(1+|\log t|)^{\alpha p}\chi_{(0,\infty)}(t).
\]

Interpolation theory provides a powerful approach to finding conditions on Lorentz-Zygmund indices that are sufficient to imply boundedness of the Fourier series map between Lorentz-Zygmund spaces. This is done in \cite{BR} and elsewhere. For example, if $1<p<\infty$, $\beta\in \mathbb R$, and $1<r<2$, the method of Remark \ref{example} may be applied with $u(t)=t^{p/r'-1}(1+|\log(t)|)^{q\beta}$ to show that the Fourier transform is bounded from $L^{r,p}(\log L)^{\beta}$ to $\ell^{r',p}(\log\ell)^\beta$. We omit the details.

Our next result complements these by providing necessary conditions. 

Observe that $L^{r,p}(\log L)^\alpha\ne\{0\}$ if and only if $t^{p/r-1}(1-\log t)^\alpha$ is integrable near zero. To avoid the trivial case we assume that 
\begin{equation}\label{*}
r<\infty,\text{ or }\big\{r=\infty,\ p<\infty,\ \alpha p<-1\big\},\text{ or }\big\{r=\infty,\ p=\infty,\ \alpha\le 0\big\}.
\end{equation}

\begin{thm} Suppose $p,q\in (0,\infty)$, $r,s\in(0,\infty]$, $\alpha,\beta\in(-\infty,\infty)$ and (\ref{*}) holds. If $\mathcal F:L^{r,p}(\log L)^\alpha\to \ell^{s,q}(\log \ell)^\beta$
then either $s>2$, or $s=2$ and $\beta\le0$. Also, either $1/r+1/s<1$, or  $1/r+1/s=1$ and $\beta\le \alpha$.
\end{thm}

\begin{proof} The hypothesis may be stated as, $\|\hat f\|_{\Lambda_q(u)}\le C\|f\|_{\Lambda_p(w)}$ for $f\in L^1(\mathbb{T})$, where
\[
w(t)=t^{p/r-1}(1-\log t)^{\alpha p}\chi_{(0,1)}(t)
\quad\text{and}\quad
u(t)=t^{q/s-1}(1+|\log t|)^{\beta q}\chi_{(0,\infty)}(t).
\]
So Corollary 3.8 implies,
\begin{equation}\label{**}
\sup_{z>1}\frac{\sup_{A\in\mathcal A}\|A\omega_z\|_{q/2,u}}{\left(\int_0^{1/z}w(t)\,dt\right)^{2/p}}<\infty.
\end{equation}
Since (\ref*) holds, the denominator is finite for each $z>1$ and it follows that the numerator is finite as well. In particular, 
$\lim_{y\to\infty}\|A_y\omega_z\|_{q/2,u}<\infty$ where,
for $y>z$, $A_y$ is the averaging operator based on the single interval $(0,y)$. For this operator,
\[
A_y\omega_z(t)
=\frac1y\left(\frac2z-\frac1y\right)\chi_{(0,y)}(t)+\frac1{t^2}\chi_{(y,\infty)}(t)
\ge\frac1{yz}\chi_{(0,y)}(t),
\]
so
\[
\lim_{y\to\infty}\left(\frac1{y^{q/2}}\int_0^y t^{q/s-1}(1+|\log t|)^{\beta q}\,dt\right)^{2/q}<\infty.
\]
From this we conclude that either $s>2$, or $s=2$ and $\beta\le0$. This proves the first conclusion of the theorem. 

For the second conclusion, take $A$ to be the identity operator in (\ref{**}). Since $z^{-2}\chi_{(0,z)}(t)\le\omega_z(t)$, 
\[
\sup_{z>1}\frac1{z^2}\left(\int_0^zt^{q/s-1}(1+|\log t|)^{\beta q}\,dt\right)^{2/q}
\left(\int_0^{1/z}t^{p/r-1}(1-\log t)^{\alpha p}\,dt\right)^{-2/p}<\infty.
\]
But an easy estimate gives,
\[
\frac d{dt}(t^{q/s}(1+|\log t|)^{\beta q})
\le t^{q/s-1}(1+|\log t|)^{\beta q}(q/s+|\beta|q)
\]
so
\[
(q/s+|\beta|q)\int_0^{z}t^{q/s-1}(1+|\log t|)^{\beta q}\,dt
\ge z^{q/s}(1+|\log z|)^{\beta q}.
\]
It follows that 
\[
\lim_{z\to\infty} \frac{z^{-2}\left(z^{q/s}(1+\log z)^{\beta q}\right)^{2/q}}{\left(\int_0^{1/z}t^{p/r-1}(1-\log t)^{\alpha p}\,dt\right)^{2/p}}<\infty.
\]
Taking the $2/p$ exponent outside the limit and applying L'Hospital's rule, gives,
\[
\lim_{z\to\infty} \frac{z^{-p+p/s-1}(1+\log z)^{\beta p}(-p+p/s+\beta p/(1+\log z))}{-z^{-2}z^{1-p/r}(1+\log z)^{\alpha p}}<\infty.
\]
Now it follows that either $1/r+1/s<1$, or  $1/r+1/s=1$ and $\beta\le \alpha$.
\end{proof}


\begin{thebibliography}{00}

\bibitem{AM} {M. Ari\~no and B. Muckenhoupt}, {Maximal functions on classical Lorentz spaces and Hardy's inequality with weights for non-increasing functions}.
Trans. Amer. Math. Soc., \textbf{320}(1990), 727--735.

\bibitem{BH1} {J. J. Benedetto and H. P. Heinig}, {Weighted Hardy spaces and the Laplace transform}, Harmonic analysis (Cortona, 1982),  240--277, 
Lecture Notes in Math., \textbf{992}, Springer, Berlin, 1983. 

\bibitem{BH2} {J. J. Benedetto and H. P. Heinig}, {Weighted Fourier inequalities: new proofs and generalizations}. J. Fourier Anal. Appl., \textbf{9}(2003), 1--37.

\bibitem{BHJ} {J. J. Benedetto, H. P. Heinig and R. Johnson}, {Weighted Hardy spaces and the Laplace transform II}, Math. Nachr. \textbf{132}(1987), 29--55.

\bibitem{BR} {C. Bennett and K. Rudnick}, {On Lorentz-Zygmund spaces}, Dissertationes Math. (Rozprawy Mat.), \textbf{175}(1980), 1--67.

\bibitem{BS} C. Bennett and R. Sharpley, {\em Interpolation of operators}. Pure and Applied Mathematics \textbf{129}, Academic Press Inc., Boston, 1988.

\bibitem{B} {J. S. Bradley}, {Hardy inequalities with mixed norms}, Canad. Math. Bull. \textbf{21}(1978), 405--408.

\bibitem{CGS} {M. Carro, A. Garc\'ia del Amo and J. Soria}, {Weak-Type Weights and Normable Lorentz Spaces.} Proc. Amer. Math. Soc., \textbf {124}(1996), 849--857.

\bibitem{JT} {M. Jodeit, Jr., and A. Torchinsky}, {Inequalities for Fourier transforms}. Studia Math., \textbf{37}(1971), 245--276.

\bibitem{M1} {L. Maligranda}, {Weighted inequalities for monotone functions}. Collect. Math., \textbf{48}(1997), 687--700.

\bibitem{M2} {L. Maligranda}, {Weighted inequalities for quasi-monotone functions}. J. London Math. Soc., \textbf{57}(1998), 363--370.

\bibitem{CCDS} {M. Mastylo and G. Sinnamon}, {A Calder\'on couple of down spaces}. J. Funct. Anal., \textbf{240}(2006), 192--225.

\bibitem{Sa} {E. Sawyer}, {Boundedness of classical operators on classical Lorentz spaces}. Studia Math., \textbf{96}(1990), 145--158.

\bibitem{LRI} {G. Sinnamon}, {The level function in rearrangement invariant spaces.} Publ. Mat., \textbf{45}(2001), 175--198.

\bibitem{ECF} {G. Sinnamon}, {Embeddings of concave functions and duals of Lorentz spaces}. Publ. Mat., \textbf{46}(2002), 489--515.

\bibitem{FTLS} {G. Sinnamon}, {The Fourier transform in weighted Lorentz spaces}. Publ. Mat., \textbf{47}(2003), 3--29.

\bibitem{TMWNI} {G. Sinnamon}, {Transferring monotonicity in weighted norm inequalities}. Collect. Math., \textbf{54}(2003), 181--216.

\bibitem{MBFS} {G. Sinnamon}, {Monotonicity in Banach function spaces}. 
NAFSA 8—Nonlinear analysis, function spaces and applications. Vol. 8,  204--240, Czech. Acad. Sci., Prague, 2007.

\bibitem{japan} {G. Sinnamon}, {Fourier inequalities and a new Lorentz space}. Banach and function spaces II,  145--155, Yokohama Publ., Yokohama, 2008.

\end{thebibliography}
\end{document}